\documentclass[10pt]{article}

\usepackage{amsmath}
\usepackage{manfnt}
\usepackage{amssymb}
\usepackage{graphicx}
\usepackage{geometry}
\usepackage{amsthm}
\usepackage[dvipsnames]{xcolor}
\usepackage[utf8]{inputenc}
\usepackage{pst-node}
\usepackage{tikz-cd} 
\usepackage{hyperref}
\usepackage{pdfpages}
\usepackage{framed}
\usepackage[backend=biber]{biblatex}
\usepackage{comment}
\usepackage[english]{babel}
\graphicspath{{graphics}}

\newcommand{\opLi}{\operatorname{Li}}

\newcommand{\opIm}{\operatorname{Im}}
\newcommand{\opNeu}{\operatorname{Neu}}
\newcommand{\opDir}{\operatorname{Dir}}

\newcommand{\oploc}{\operatorname{loc}}

\newcommand{\opRange}{\operatorname{Range}}

\newcommand{\opsupp}{\operatorname{supp}}

\newcommand{\opRe}{\operatorname{Re}}
\newcommand{\opspec}{\operatorname{spec}}

\newcommand\numberthis{\addtocounter{equation}{1}\tag{\theequation}}

\newtheorem{Theorem}{Theorem}
\newtheorem{Proposition}[Theorem]{Proposition}
\newtheorem{Remark}[Theorem]{Remark}
\newtheorem{Definition}[Theorem]{Definition}
\newtheorem{Lemma}[Theorem]{Lemma}

\newtheorem{Corollary}[Theorem]{Corollary}

\title{Exact output tracking for the one-dimensional heat equation and applications to the interpolation problem in Gevrey classes of order 2}
\author{Lucas Davron\thanks{CEREMADE, Universit\'e Paris-Dauphine \& CNRS UMR 7534, Universit\'e PSL, 75016 Paris, France ({davron@ceremade.dauphine.fr)}}\quad and Pierre Lissy\thanks{CERMICS, Ecole des Ponts, IP Paris, Marne-la-Vall\'ee, France ({pierre.lissy@enpc.fr})}}

\begin{document}
\maketitle
\paragraph*{Abstract}
This paper provides a complete characterization of the Dirichlet boundary outputs that can be exactly tracked in the one-dimensional heat equation with Neumann boundary control. The problem consists in describing the set of boundary traces generated by square-integrable controls over a finite or infinite time horizon. We show that these outputs form a precise functional space related to Gevrey regularity of order~2. In the infinite-time case, the trackable outputs are precisely those functions whose successive derivatives satisfy a weighted summability condition, which corresponds to specific Gevrey classes. For finite-time horizons, an additional compatibility condition involving the reachable space of the system provides a full characterization. The analysis relies on Fourier-Laplace transform, properties of Hardy spaces, the flatness method, and a new Plancherel-type theorem for Hilbert spaces of Gevrey functions. Beyond control theory, our results yield an optimal solution to the classical interpolation problem in Gevrey-$2$ classes, which improves results of Mitjagin on the optimal loss factor. The techniques developed here also extend to variants of the heat system with different boundary conditions or observation points.


\section{Introduction}
\subsection{Output tracking}
Let us consider the one-dimensional heat problem with Neumann boundary conditions
\begin{equation}\label{eq:heat}
\left\lbrace \begin{array}{rclcc}
z_t(t,x) &=& z_{xx}(t,x),& 0 < x < 1,&t>0,\\
z_x(t,1) &=& u(t), \\
z_x(t,0) &=& 0,\\
z(0,x) &=& 0 ,  \\
y(t) &=& z(t,0),
\end{array}\right.
\end{equation}
where $z = z(t,x)$ denotes the unknown state, $u = u(t)$ the control, and $y = y(t)$ is what we call the \textit{output signal}. In the present article, we aim to obtain a sharp characterization of these outputs $y$ when $u$ ranges over $L^2(0,T)$. In other words, we look for a description of the set 
\[
\mathcal{Y}(0,T):= \{ y \in L^1_{\oploc}[0,T): u \in L^2(0,T) \},
\]
where $T \in (0,\infty]$ is the horizon time. We refer to this as an \emph{exact output tracking problem.}

In general, output tracking is  a central problem in automatic control theory. From both theoretical and practical perspectives, the full state of a control system is often neither directly controllable nor observable, making full-state tracking unrealistic. In such situations, one seeks instead to track only a given \emph{output} of the system, such as, a boundary trace, for instance, as in the present example.

Furthermore, output tracking can be viewed as a preliminary step towards the more demanding problem of \emph{output regulation}, which also involves disturbance rejection and robustness considerations. Most of the existing literature deals with asymptotic tracking achieved through (possibly dynamic) feedback laws, where the controlled output asymptotically approaches a reference trajectory as $t \to +\infty$ (see, for instance, \cite[Chapter 12]{back} for related results on the one-dimensional heat equation).

In contrast, much less is known about \emph{exact} or \emph{approximate} output tracking achieved through open-loop controls, both for finite-dimensional linear systems and for controlled partial differential equations. Some recent progress in this finite-dimensional setting was reported in \cite{zzz25}. In the context of the wave equation, exact output tracking in a closely related framework was investigated in \cite{Z1,Z2}. Finally, for higher-dimensional generalizations of system~\eqref{eq:heat}, the weaker notion of \emph{approximate output tracking}—where the desired tracking property is achieved up to an arbitrarily small error using an open-loop control—has been addressed in \cite{MR4873205}.

Concerning the exact tracking of \eqref{eq:heat}, only few results already exist in the literature, which are all in the case of finite $T$. On the one hand, any $y \in \mathcal{Y}(0,T)$ satisfies $y \in C^\infty[0,T]$ together with the two conditions 
\begin{equation}\label{eq:flat}
y^{(k)}(0) = 0,\quad \forall k \in \mathbb{N},
\end{equation}
and
\begin{equation}\label{eq:gevrey}
\exists C,R > 0,\quad \forall t \in [0,T],\quad \forall k \in \mathbb{N},\quad |y^{(k)}(t)| \leq C \frac{(2k)!}{R^{2k}}.
\end{equation}
Such a regularity result can be deduced from Gevrey's estimates \cite[Proposition 12.1, Chapter 5]{benedetto}. Moreover, in the above, the parameter $R$ is uniform with respect to $T$ and $y$. On the other hand, it has been established by Martin Rosier and Rouchon in \cite[Proposition 3.1]{reachable_martin} (see also Hölmgren \cite{holmgren} for connected results) using the flat-output method that a sufficient condition for a signal $\varphi \in C^\infty[0,T]$ to belong to $\mathcal{Y}(0,T)$ is that there is some $R>1$ such that $\varphi$  satisfies \eqref{eq:flat} together with \eqref{eq:gevrey}. \par 
Let us denote by $G^{2,R}[0,T]$ the set of those $\varphi \in C^\infty[0,T]$ satisfying the estimate \eqref{eq:gevrey}, and $G^{2,R}_{(0)}[0,T]$ the set of those $\varphi \in G^{2,R}[0,T]$ which satisfy in addition \eqref{eq:flat}. The above two results can be summed-up as 
\[
\exists R \leqslant  1,\quad \forall \epsilon,T>0,\quad G^{2,1+\epsilon}_{(0)}[0,T] \subset \mathcal{Y}(0,T) \subset G^{2,R}_{(0)}[0,T].
\]
Our description of $\mathcal{Y}(0,T)$ will use Hilbert spaces of Gevrey functions similar to $G^{2,R}$. We start with the case $T = +\infty$, which is our first main result.
\begin{Theorem}\label{theo:charac_Y_Tinfty}
The set $\mathcal{Y}(0,\infty)$ consists of those $\varphi \in C^\infty_{(0)}[0,\infty)$ such that 
\begin{equation}\label{eq:charac_Y_gevrey}
    \sum_{k=0}^{\infty} \left( \frac{||{\varphi^{(k+1)}}||_{L^2(0,\infty)}}{(2k)!2^k(1+k)^{3/4}}\right)^2 < \infty. 
\end{equation}
\end{Theorem}
The estimate \eqref{eq:charac_Y_gevrey} essentially means that $\varphi$ is Gevrey of order 2 and radius $1/\sqrt{2}$. The constant $1/\sqrt{2}$ is no coincidence, and could have been guessed from the characterization of the reachable space of the heat equation and the flat-output method. We introduce these two theories below. 

The reachable space for \eqref{eq:heat} is defined by
\begin{equation}\label{def:rs}
\mathcal{R}_T:= \{ z(T,\cdot): u \in L^2(0,T) \},\quad (T < \infty).
\end{equation}
The set $\mathcal{R}_T$ has been studied in great detail in \cite{fattorini_reachable,hartmann_orsoni,seidman, reachable_martin,sharp_reachable,reachable_heat,rkhs} where an explicit characterization was obtained. As shown in the cited references, the set $\mathcal{R}_T$ does not depend on $T$; we denote it by $\mathcal{R}$. Moreover, it consists of those functions $f$ that are holomorphic and even on the tilted square 
\[
\Omega:= \{ \zeta=a+ib \in \mathbb{C}: |a| + |b| < 1 \},
\]
and such that, moreover, $f' \in L^2(\Omega)$, where $f' = \partial f / \partial \zeta$ is the complex derivative. Note that such an $f$ satisfies 
\[
\frac{\partial f}{\partial {\bar{\zeta}}} = 0,\quad \frac{\partial f}{\partial \zeta} \in L^2(\Omega),
\]
hence, regarding $f$ as a function of two real variables we have $\nabla f \in L^2(\Omega)$. Therefore, by standard procedures (see \textit{e.g.} \cite[\S 1.1.11]{mazya}), we deduce that $f$ belongs to the Sobolev space $H^1(\Omega)$, where $\Omega$ is seen as a subset of $\mathbb R^2$. Thus,
\begin{equation}\label{eq:reachable_neu}
    \mathcal{R} = \{  f \in \mathcal{H}(\Omega)\cap H^1(\Omega),\, f \mbox{ even} \},
\end{equation}
where $\mathcal{H}(D)$ stands for the set of holomorphic functions on $D$, whenever $D$ is an open subset of $\mathbb{C}$.

Now take a solution $z$ of \eqref{eq:heat}. It lies in $C^\infty((0,T) \times (0,1))$, and can be extended as an even function of $x$, lying in $C^\infty((0,T) \times (-1,1))$, which solves the heat equation on $(0,T) \times (-1,1)$. The flat-output method is a general procedure to parameterize the state of a system as a power series of the successive time derivatives of a well-chosen output, we refer to \cite{fliess} for a general overview. For the heat equation \eqref{eq:heat}, this method has been implemented by developing $z$ as power series of $x$, we refer to \cite{holmgren, laroche, reachable_martin}. Notably, for any $u \in L^2(0,T)$, the solution $z$ of \eqref{eq:heat} writes
\begin{equation}\label{plat1}
z(t,x) = \sum_{k=0}^\infty y^{(k)}(t) \frac{x^{2k}}{(2k)!},
\end{equation}
where for fixed $0 \leq t \leq T$, the power series converges for $|x|$ small enough. 

From \eqref{plat1} and Cauchy's inequalities, we see that the largest $R$ such that $y \in G^{2,R}[0,T]$ is also the largest $R$ such that $z(t,x)$ is holomorphic with respect to $x$ in $D(0,R)$. From \eqref{eq:reachable_neu}, it follows that the largest such $R$ is $R = 1/\sqrt{2}$, therefore, trackable functions are Gevrey of order 2 and radius $R \leq 1/\sqrt{2}$, and no better in general. Theorem \ref{theo:charac_Y_Tinfty} makes this more precise and more importantly provides a converse statement.  \newline
\newline
Let us now turn to the case $0 < T < \infty$. Clearly, any $y \in \mathcal{Y}(0,T)$ has an extension to $\mathcal{Y}(0,\infty)$ (extend the control $u$ by zero for times $t > T$), which has to satisfy \eqref{eq:charac_Y_gevrey}. Therefore, a necessary condition for $y \in \mathcal{Y}(0,T)$ is that 
\begin{equation}\label{eq:reg_Y_T_fini}
    \sum_{k=0}^{\infty} \left( \frac{||{y^{(k+1)}}||_{L^2(0,T)}}{(2k)!2^k(1+k)^{3/4}}\right)^2 < \infty. 
\end{equation}
From \eqref{plat1} and \eqref{def:rs}, another necessary condition for $y$ to belong to $\mathcal{Y}(0,T)$ is that 
\begin{equation}\label{eq:final_in_R}
    \sum_{k=0}^\infty y^{(k)}(T) \frac{x^{2k}}{(2k)!} \in \mathcal{R}.
\end{equation}
It turns out that the two conditions  \eqref{eq:reg_Y_T_fini} and \eqref{eq:final_in_R} are also sufficient. 
\begin{Theorem}\label{theo:Y_finite_T}
For any $0 < T < \infty$, the space $\mathcal{Y}(0,T)$  consists of those $y \in C^\infty_{(0)}[0,T]$ which satisfy the two conditions \eqref{eq:reg_Y_T_fini} and \eqref{eq:final_in_R}. 
\end{Theorem}
Using cutoff functions that are Gevrey of order $1 < s < 2$, we easily deduce the following.
\begin{Corollary}\label{coro:loc_tracable}
Let $0 < T < \infty $ and $\varphi \in C^\infty_{(0)}[0,T]$ satisfying \eqref{eq:reg_Y_T_fini} (where $y$ is replaced by $\varphi$). Then for every $0 < \delta < T$, we have $\varphi \in \mathcal{Y}(0,T-\delta)$.
\end{Corollary}
The condition \eqref{eq:reg_Y_T_fini} does not imply \eqref{eq:final_in_R}; the following proposition provides a counterexample. Here and throughout, the symbol $\lesssim $ means the inequality holds up to a numerical multiplicative constant independent of the parameters; a similar meaning is used for $\gtrsim$.

\begin{Proposition}\label{prop:borel_LR}
There exists a sequence of real non-negative numbers $(a_k)_{k\in\mathbb N}$ such that: 
\begin{itemize}
    \item There exists $\varphi \in C^\infty[0,1]$ such that
\begin{equation}\label{ce1}\varphi^{(k)}(1^-) = a_k,\quad |\varphi^{(k)}(t)| \lesssim \frac{(2k)!}{\sqrt{1+k}},\quad (k \geq 0,\quad 0 \leq t \leq 1).
\end{equation}
\item There is no $\psi \in C^\infty[1,2]$ satisfying \eqref{eq:reg_Y_T_fini} (with $(1,2)$ in place of $(0,T)$) such that 
\begin{equation}\label{eq:interp_bn_right}
\psi^{(k)}(1^+) = a_k,\quad \forall k \geq 0. 
\end{equation}
\end{itemize}
\end{Proposition}

\begin{Remark}\label{rem:MMR}
We notice that Proposition \ref{prop:borel_LR} contradicts \cite[Theorem 3.2]{reachable_martin}, see Appendix~\ref{app}. 
\end{Remark}

\subsection{Some intermediate results of interest}
Let us sketch the proof of Theorem \ref{theo:charac_Y_Tinfty}, in order to emphasize two intermediate results that are interesting in themselves (namely, Proposition \ref{prop:hardy} and Theorem \ref{theo:plancherel_gevrey}). For $u \in L^2(0,\infty)$, we apply the Laplace transform with respect to time in \eqref{eq:heat} to obtain
\[
\hat{y}(s) = \frac{\hat{u}(s)}{\sqrt{s} \sinh \sqrt{s}},\quad s \in \mathbb{C}_+,
\]
where $\hat{~}$ stands for the Laplace transform and $\mathbb{C}_+$ is the open right half-plane. This can be rewritten as
\begin{equation}\label{eq:laplace_rewrite}
\hat{u}(s) = \phi(s)\hat{\dot{y}}(s),\quad \phi(s):= \frac{\sinh \sqrt{s}}{\sqrt{s}},
\end{equation}
where $\hat{u}$ belongs to the Hardy space $\mathcal{H}^2(\mathbb{C}_+)$ and $1/\phi \in \mathcal{H}^\infty(\mathbb{C}_+)$, hence $\hat{\dot{y}} \in \mathcal{H}^2(\mathbb{C}_+)$ (see \S \ref{subsec:fine_hardy} for the definition of Hardy spaces). From \eqref{eq:laplace_rewrite} we see that the description of $\mathcal{Y}(0,\infty)$ is equivalent, through the Laplace transform, to the description of these $F \in \mathcal{H}^2(\mathbb{C}_+)$ such that $\phi F$ lies again in $\mathcal{H}^2(\mathbb{C}_+)$. Proposition \ref{prop:hardy} below allows us to characterize such functions $F$, under a growth assumption on $\phi$. 
\begin{Proposition}\label{prop:hardy}
Let $\phi$ be an entire function of order\footnote{see Definition \ref{order}} $< 1$. Then, for every $F \in \mathcal{H}^2(\mathbb{C}_+)$, the product $\phi F$ lies in $\mathcal{H}^2(\mathbb{C}_+)$ if and only if 
\[
\int_\mathbb{R} | \phi(i\xi) F(i\xi) |^2 d\xi < \infty. 
\]
\end{Proposition}
We will obtain Proposition \ref{prop:hardy} as a consequence of the following assertion:
\begin{equation}\label{eq:spec_entire}
    \opsupp \mathcal{F} \phi(i\cdot) \subset \{ 0 \}. 
\end{equation}
On the one hand, a function $\phi$ entire on $\mathbb{C}$ with order $< 1$ may not even lie in $\mathcal{S}'(\mathbb{R})$ (\textit{e.g.} $\phi(s) = \cosh \sqrt{s}$), hence \eqref{eq:spec_entire} is a formal statement. On the other hand, the Paley--Wiener theorem \cite[Theorem 19.3, p. 370]{rudin} asserts that if $f \in \mathcal{H}(\mathbb{C}) \cap L^2(\mathbb{R})$ satisfies
\[
    \exists C > 0,\quad \forall z \in \mathbb{C},\quad |f(z)| \leq Ce^{A|z|},
\]
for some $A > 0$, then $\opsupp \mathcal{F}f \subset [-A,A]$. As a function $\phi$ of order $<1$ satisfies the above estimate for any $A > 0$, the above Paley-Wiener theorem is somehow compatible with \eqref{eq:spec_entire}. Lemma \ref{lem:spectrum} will give a meaning to \eqref{eq:spec_entire}. \newline
\newline
From Proposition \ref{prop:hardy} and \eqref{eq:laplace_rewrite}, we readily obtain that 
\begin{equation}\label{eq:charac_Y_fourier}
\int_\mathbb{R} \left| \mathcal{F}\dot{y}(\xi) \omega(\xi) \right|^2 d\xi < \infty,\quad \omega(\xi) = \frac{e^{\sqrt{|\xi|/2}}}{(1+|\xi|)^{1/2}},
\end{equation}
characterizes $\mathcal{Y}(0,\infty)$. To translate \eqref{eq:charac_Y_fourier} in the time domain, we will establish the following Plancherel-type theorem for certain Hilbert spaces of Gevrey functions. For parameters $s,R > 0$ and $\gamma \in \mathbb{R}$ define the Hilbert space
\[
\hat{\mathcal{G}}_{s,R,\gamma}:= L^2(\mathbb{R} , \omega(\xi)^2d\xi),\quad  \omega(\xi) = (1+|\xi|)^\gamma e^{R |\xi|^{1/s}}.
\]
Note that $\hat{\mathcal{G}}_{s,R,\gamma} \subset L^2(\mathbb{R},d\xi)$ with continuous dense injection. The Fourier transforms of the elements of $\hat{\mathcal{G}}_{s,R,\gamma}$ turn out to be precisely those $\varphi \in C^\infty(\mathbb{R})$ such that 
\begin{equation}\label{eq:def_G_sRgamma_Mn}
    \| \varphi\|_{\mathcal{G}_{s,R,\gamma}}^2:=  \sum_{n=0}^\infty \left( \frac{\| \varphi^{(n)}\|_{L^2(\mathbb{R})}}{M_n} \right)^2 < \infty , \quad M_n:=  \frac{(ns)!}{R^{ns}} (1+n)^{-s\gamma-1/4},
\end{equation}
where $(ns)! := \Gamma(ns+1)$. The above defines a Hilbert space $\mathcal{G}_{s,R,\gamma}$, which is continuously and densely embedded in $L^2(\mathbb{R},dx)$.
\begin{Theorem}\label{theo:plancherel_gevrey}
For every $s ,R > 0$ and $\gamma \in \mathbb{R}$, the Fourier transform is a topological isomorphism between $\mathcal{G}_{s,R,\gamma}$ and $\hat{\mathcal{G}}_{s,R,\gamma}$.
\end{Theorem}
\begin{Remark}
As will be clear from the proof of Theorem \ref{theo:plancherel_gevrey}, there exists a sequence $(\theta_n)$ of positive numbers satisfying
\[
\exists 0 < \underline{\theta} \leq \overline{\theta} < \infty,\quad \forall n \in \mathbb{N},\quad \underline{\theta} \leq \theta_n \leq \overline{\theta},
\]
and such that if we replace $M_n$ by $\theta_n M_n$ in the definition of $\mathcal{G}_{s,R,\gamma}$, then the Fourier transform is an isometry $\mathcal{G}_{s,R,\gamma} \rightarrow \hat{\mathcal{G}}_{s,R,\gamma}$. This fact will however not be used, and it will be more convenient for us to work with $M_n$ given in \eqref{eq:def_G_sRgamma_Mn}. 
\end{Remark}
To our knowledge, Theorem \ref{theo:plancherel_gevrey} is new even for $\gamma=0$. Indeed, what is well-established is that a function $\varphi \in C_c^\infty(\mathbb{R})$ such that
\[
\exists R ,C > 0,\quad \forall k \in \mathbb{N},\quad \forall x \in \mathbb{R},\quad |\varphi^{(k)}(x)| \leq C \frac{(ns)!}{R^{ns}},
\]
will have a Fourier transform growing like
\[
\exists L,C' > 0,\quad \forall \xi \in \mathbb{R},\quad |\mathcal{F}\varphi(\xi)| \leq C' e^{-L|\xi|^{1/s}},
\]
and the converse also holds (see \textit{e.g.} \cite[Chapter IV]{gs2},\cite[Theorem 1.6.1]{Rodino}). In the previous works we are aware of, a link between $R$ and $L$ can be made explicit, but it is never sharp.
\subsection{Consequences for the interpolation problem in Gevrey-$2$ class}
We now consider the interpolation problem in Gevrey classes\footnote{We refer to \cite{bilodeau} for a historical survey. For more results on the interpolation problem, notably on more general Denjoy-Carleman classes, see \cite{wahde, petzsche, kiro}.} of order $s > 1$. For fixed $s > 1$, we consider a sequence $(a_n)$ of complex numbers such that
\begin{equation}\label{eq:sequence_gevrey}
    \exists R_0,C > 0,\quad \forall n \in \mathbb{N},\quad |a_n| \leq C \frac{(ns)!}{R_0^{ns}}.
\end{equation}
We say that a function $\varphi \in C^\infty(\mathbb{R})$ interpolates $(a_n)$ (at $t=0$) if $\varphi^{(n)}(0) = a_n$ for every $n$. From Borel's Theorem\footnote{Borel published this result in 1895 within his thesis \cite{Bobo}, although it seems to have been proved first by Peano in 1884, see \cite{borel}.} we know that such a function $\varphi$ exists, regardless of the growth hypothesis made on $(a_n)$. In \cite{carleson} Carleson studies the universal moment problem and establishes a sufficient condition for it to be solved. The moment problem is connected with the interpolation problem through the formula 
\[
\varphi^{(n)}(0) = \frac{1}{\sqrt{2 \pi}} \int_\mathbb{R} \xi^n \mathcal{F}\varphi(-\xi) d\xi,\quad \mathcal{F}\varphi(\xi) := \frac{1}{\sqrt{2 \pi}} \int_\mathbb{R} e^{-it\xi} \varphi(t) dt.
\]
As a consequence of his result on the moment problem, Carleson obtains that one may always interpolate a sequence $(a_n)$ satisfying \eqref{eq:sequence_gevrey} by a function $\varphi$ which is Gevrey of order $s$, \textit{i.e.} satisfying 
\begin{equation}\label{eq:interpolate_gevrey}
    \exists R_1,C' > 0,\quad \forall n \in \mathbb{N},\quad \forall t \in \mathbb{R},\quad  |\varphi^{(n)}(t)| \leq C' \frac{(ns)!}{R_1^{ns}}.
\end{equation}
For the applications in control theory we have in mind \cite{reachable_martin,rosier, laurent}, understanding how the parameter $R_1$ in \eqref{eq:interpolate_gevrey} depends on the parameter $R_0$ in \eqref{eq:sequence_gevrey} is crucial in order to obtain sharp reachability results. Around the same year Carleson's results are published, Mitjagin \cite{mitjagin} independently shows that one can always take $R_1 = \rho_s R_0 - \epsilon$, where $\rho_s:= \cos(\pi/2s) < 1$ and $\epsilon$ is any prescribed positive number. In \cite{mitjagin}, the interpolating function $\varphi$ is defined on the smaller interval $[-1,1]$, which is not essential, as can be seen by cut-off arguments, see \cite[Lemma 3.7]{reachable_martin}.  We shall refer to the numerical constant $\rho_s$ as the loss factor of the interpolation problem, as it prescribes how one should relax the growth assumptions when passing from \eqref{eq:sequence_gevrey} to \eqref{eq:interpolate_gevrey}. The constant $\rho_s$ was further shown by Mitjagin to be sharp in the sense that for every $\epsilon , R_0 > 0$ and $s > 1$, there exists a sequence $(a_n)$ satisfying \eqref{eq:sequence_gevrey} which can be interpolated by no function $\varphi$ satisfying \eqref{eq:interpolate_gevrey} with $R_1 \geq \rho_s R_0 + \epsilon$. Mitjagin's proof is quite elegant, and the optimality of the constant $\rho_s$ is deduced from its optimality in a Phragmén-Lindelöf principle. The paper \cite{mitjagin} was originally written in Russian, and its translation is rather hard to access. Furthermore, in view of its size (4 pages), the published version seems to offer only a sketch of the proof rather than a fully detailed demonstration. Hence, for the sake of completeness, we give in Appendix \ref{subsec:mit} a detailed proof of Mitjagin's solution of the interpolation problem. 

As a consequence of our study of the trackable subspace $\mathcal{Y}(0,T)$ we can make more precise the relation between $(a_n)$ and $\varphi$, when $s=2$. The following result characterizes the range of the Borel operator $\mathcal{B}_{t_0}$, defined on smooth functions by $\mathcal{B}_{t_0}\varphi := (\varphi^{(n)}(t_0))_{n=0}^\infty$, over Hilbert spaces of Gevrey functions. We will use the following notation: for an interval $I \subset \mathbb{R}$ and parameters $s,R > 0$ and $\gamma \in \mathbb{R}$, we let $\mathcal{G}_{s,R,p}(I)$ be the Hilbert space defined similarly as in \eqref{eq:def_G_sRgamma_Mn}, with $L^2(I)$ in place of  $L^2(\mathbb{R})$.
\begin{Proposition}\label{prop:borel_gevrey}
Let $I \subset \mathbb{R}$ be a non trivial interval and $t_0$ lying in its interior. Let also $R > 0$ and $p \in \mathbb{Z}$. Then 
\[
\mathcal{B}_{t_0} \mathcal{G}_{2,R,p}(I) = \left\lbrace (a_k): \sum^\infty a_{k+p} \frac{(\sqrt{2}R\zeta)^{2k}}{(2k)!} \in A^2(\Omega) \right\rbrace.
\]
and 
\[
\mathcal{B}_{t_0} \mathcal{G}_{2,R,p-1/2}(I) = \left\lbrace (a_k): \sum^\infty a_{k+p} \frac{(\sqrt{2}R\zeta)^{2k+1}}{(2k+1)!} \in A^2(\Omega) \right\rbrace.
\]
\end{Proposition}
In the above, $A^2(\Omega)$ stands for the Bergman space on $\Omega$, that is the set of functions which are holomorphic and square integrable on $\Omega$. 

For an arbitrary sequence $(a_n)$, we deduce a formula for the largest $R > 0$ such that $(a_n)$ can be interpolated by a function $\varphi$ belonging to the set $G^{2,R}[-1,1]$, which is defined similarly as in \eqref{eq:gevrey} with $[-1,1]$ in place of $[0,T]$. To state the result we introduce the notation
\[
R_a:=  \sup \{ R > 0: \exists \varphi \in G^{2,R}[-1,1],\quad \forall n \in \mathbb{N},\quad \varphi^{(n)}(0) = a_n \},
\]
where it is understood that $R_a = 0$ if the above supremum runs over the empty set.
\begin{Corollary}\label{coro:individual_loss}
Let $a=(a_k)$ be an arbitrary sequence. Then 
\[
R_a = \sup \left\lbrace R \geq 0: \sum_{k=0}^\infty a_k \frac{(\sqrt{2} R \zeta)^{2k}}{(2k)!} \in A^2(\Omega) \right\rbrace.
\]
\end{Corollary}
The rest of the paper is organized as follows. In Section 2 we show our tracking results Theorems \ref{theo:charac_Y_Tinfty} and \ref{theo:Y_finite_T}, as well as the intermediate results Proposition \ref{prop:hardy} and Theorem \ref{theo:plancherel_gevrey}, and finally Proposition \ref{prop:borel_LR}. In Section 3 we explain how our methods can be applied to systems close to \eqref{eq:heat}, having different boundary conditions or smoother control laws, or where the output is given by the Dirichlet trace at a given point of $(0,1)$. In Section 4 we show Proposition \ref{prop:borel_gevrey} and Corollary \ref{coro:individual_loss}. In Appendix \ref{app} we study the interpolation problem in Gevrey classes of order $s > 1$, notably by providing a detailed proof Mitjagin's \cite[Theorem 1 and 1a]{mitjagin}.

\paragraph*{Acknowledgements}
The authors are deeply indebted to Oliver Glass for many valuable discussions and insightful suggestions. Lucas Davron thanks Sylvain Ervedoza for fruitful discussions and having suggested to use the Fourier-Laplace transform. The authors also thank Swann Marx for reading a preliminary version of this paper. Pierre Lissy was supported by the french ``Agence Nationale de la Recherche'' under the grant ANR-22-CPJ2-0138-01.
\section{Exact tracking}\label{sec:tracking}
To show Proposition \ref{prop:hardy} we will rely on the theory of boundary values of holomorphic functions, which we recall below. 
\subsection{Boundary values of holomorphic functions}\label{subsec:fine_hardy} 	
Let us recall some well-known facts on holomorphic functions in the open right half-plane.
\begin{Definition}\label{def:H2}
Let $p\in [1,+\infty[$. A function $F \in \mathcal{H}(\mathbb{C}_+)$ is said to belong to $\mathcal{H}^p(\mathbb{C}_+)$ if 
\[
\| F \|_{\mathcal{H}^p(\mathbb{C}_+)}^p:= \sup_{x > 0} \int_\mathbb{R} |F(x+i\xi)|^p d\xi < \infty.
\]
A function $F \in \mathcal{H}(\mathbb{C}_+)$ is said to belong to $\mathcal{H}^\infty(\mathbb{C}_+)$ if 
\[
\| F \|_{\mathcal{H}^\infty(\mathbb{C}_+)}:= \sup_{z \in \mathbb{C}_+}  |F(z)|< \infty.
\]
\end{Definition}
The above definition naturally yields a Banach space for any $p\in [1,+\infty]$. The following result is due to Paley and Wiener. In particular, it implies that $\mathcal{H}^2(\mathbb{C}_+)$ is also a Hilbert space.
\begin{Theorem}{\cite[Theorem 19.2]{rudin}} \label{theo:paley_wiener_isom}
Consider the Laplace transform 
\[
\mathcal{L}: L^2(0,\infty) \rightarrow \mathcal{H}^2(\mathbb{C}_+),\quad \mathcal{L}f(s) = \int_0^\infty e^{-st}f(t)dt.
\]
Then the operator $\mathcal{L}/(2\pi)$ is a surjective isometry.
\end{Theorem}
\begin{Definition}
Let $F : \mathbb{C}_+ \rightarrow \mathbb{C}$ holomorphic and $f : \mathbb{R} \rightarrow \mathbb{C}$ measurable. We say that $F$ has the non tangential (or angular) boundary value $f$ if there exists a constant $0 < \alpha < 1$ and a subset $E \subset \mathbb{R}$ of positive measure such that 
\[
\forall \xi \in E,\quad F(z) \xrightarrow[z \rightarrow i\xi]{z \in \Gamma_\alpha(i\xi)} f(\xi). 
\]
\end{Definition}
In the above, we have used the notation 
\[
\Gamma_\alpha(i\xi) = i\xi + \Gamma_\alpha,\quad \Gamma_\alpha = \{ x+i \tau \in \mathbb{C}_+: |\tau| < \alpha x \}.
\]
The set $E$ is authorized to merely be of positive measure from Luzin-Privalov's uniqueness theorem, which is stated below. This result, as well as Fatou's Theorem \ref{theo:Fatou} below, is usually proved first for functions holomorphic on the unit disc and then exported to other geometries by conformal mapping. 
\begin{Theorem}{\cite[III.D]{koosis}}
Let $F \in \mathcal{H}(\mathbb{C}_+)$, and assume that $F$ has angular boundary value zero. Then $F$ is identically zero on $\mathbb{C}_+$. 
\end{Theorem}
We will frequently abuse notations and write $F(i\xi)$ for $f(\xi)$. Note also that the function $f$ may be defined up to a zero measure set. 

The next result is an existence theorem for boundary values, due to Fatou. 
\begin{Theorem}{\cite[Theorem 17.10]{rudin}}\label{theo:Fatou}
Let $F \in \mathcal{H}^2(\mathbb{C}_+)$ and $0 < \alpha  < 1$. Then for almost every $\xi \in \mathbb{R}$, the limit 
\[
\lim_{\Gamma_\alpha(i\xi) \ni z \rightarrow i\xi} F(z) =: F(i\xi),
\]
exists. The induced function $\xi \mapsto F(i\xi)$ is $L^2(\mathbb{R})$. 
\end{Theorem}
The next two results are also known as Paley--Wiener theorems. Here and throughout, we define the Fourier transform of a function $f$ by 
\[
\mathcal{F}f(\xi) := \frac{1}{\sqrt{2 \pi}} \int_\mathbb{R} e^{-ix\xi}f(x)dx.
\]
\begin{Theorem}{\cite[Theorem 19.3]{rudin}}
Let $f \in L^2(-A,A)$ for some $A > 0$. Then $\mathcal{F}f \in \mathcal{H}(\mathbb{C}) \cap L^2(\mathbb{R})$ satisfies
\[
\exists C > 0,\quad \forall z \in \mathbb{C},\quad |\mathcal{F}f(z)| \leq C e^{A|z|}.
\]
Conversely, if $\phi \in \mathcal{H}(\mathbb{C}) \cap L^2(\mathbb{R})$ satisfies the above estimate, then $\mathcal{F} \phi$ has support contained in $[-A,A]$.  
\end{Theorem}
\begin{Theorem}{\cite[Theorem 7.2]{Katznelson}}\label{theo:paley_wiener_Hardy}
Let $f \in L^2(\mathbb{R})$. Then the two following conditions are equivalent 
\begin{enumerate}
\item There exists a function $F \in \mathcal{H}^2(\mathbb{C}_+)$ such that 
\[
F(x+i\cdot) \xrightarrow[x \rightarrow 0^+]{L^2(\mathbb{R})} f,
\]
\item $\opsupp \mathcal{F}f \subset (-\infty,0]$
\end{enumerate}
\end{Theorem}
It can be shown that the boundary value $f$ given by the previous theorem coincides with the one of Fatou's Theorem: $F(i\xi) = f(\xi)$ for almost every $\xi \in \mathbb{R}$. We will use angular boundary values rather than $L^2$ boundary values, as the former  are conveniently stable by multiplication. 

Combining Theorem \ref{theo:paley_wiener_Hardy} and Lusin-Privalov's uniqueness theorem, we obtain the following characterization of the holomorphic functions belonging to $\mathcal{H}^2(\mathbb{C}_+)$, from their boundary value. 
\begin{Proposition}\label{prop:PWP}
Let $F \in \mathcal{H}(\mathbb{C}_+)$ and $f : \mathbb{R} \rightarrow \mathbb{C}$ measurable, possibly defined up to a zero measure set. Assume that $F$ has angular boundary value $f$. Then $F \in \mathcal{H}^2(\mathbb{C}_+)$ if and only if $f$ satisfies both 
\[
f \in L^2(\mathbb{R}), \quad \text{and} \quad \opsupp \mathcal{F} f \subset (-\infty,0]. 
\]
\end{Proposition}
\subsection{Spectrum of entire functions}\label{sec:spectrum}
We are now in position to prove Proposition \ref{prop:hardy}. We will need the following definition.
\begin{Definition}{\cite[Chapter 2]{boas}} \label{order} 
The order $\rho$ of an entire non-constant function $f$ is defined as 
\[
\rho = \limsup_{r\to\infty}\frac{\log \log M_f(r) }{\log r},\quad M_f(r) := \max_{|z| \leq r} |f(z)|.
\]
If $\rho < \infty$, the type $A$ of $f$ is defined by 
\[
A = \limsup_{r\to\infty}\frac{\log M_f(r) }{r^\rho}.
\]
\end{Definition}
In particular, if both $\rho$ and $A$ are finite, we have 
\[
\forall \epsilon > 0,\quad \exists C > 0,\quad \forall z \in \mathbb{C},\quad |f(z)| \leq e^{(A + \epsilon)|z|^\rho},
\]
and $A$ is the best such constant. \newline
\newline
From Proposition \ref{prop:PWP}, Proposition \ref{prop:hardy} is a consequence of Lemma \ref{lem:spectrum} below. 
\begin{Lemma}\label{lem:spectrum}
Let $\phi$ be entire on $\mathbb{C}$ and of order $< 1$. Then for every $f \in L^2(\mathbb{R})$ such that $\phi f \in L^2(\mathbb{R})$, we have 
\[
\opsupp \mathcal{F}(\phi f) \subset \opsupp \mathcal{F}f.
\]
\end{Lemma}
Lemma \ref{lem:spectrum} is formally shown with the computations
\begin{align*}
    \opsupp \mathcal{F}[\phi f] &= \opsupp \mathcal{F}\phi * \mathcal{F}f \\
    &\subset (\opsupp \mathcal{F}\phi) + (\opsupp \mathcal{F}f) \\
    &\subset (-\infty,0] + \{ 0 \} \subset (-\infty,0], 
\end{align*}
which uses the formal assertion $\opsupp \mathcal{F}\phi \subset \{ 0 \}$. We can give a meaning to $\mathcal{F} \phi$ using Gelfand--Silov's generalized function space $(S_1^1)'$ or the theory of hyperfunctions, see \cite{gs2,kaneko}. However, in this setting it is not clear for us how to define the convolution $\mathcal{F}\phi * \mathcal{F}f$, which prevents us to use the associated Paley-Wiener theorem \cite[Théorème 1, \S III.1, p. 98]{roumieu}, \cite[Chapter 2, p. 259]{gl_eng}. 
\begin{proof}[Proof of Lemma \ref{lem:spectrum}] We let $s > 1$ to be adjusted later on, and we introduce a Gevrey approximation of the unity as follows. We start by considering a function $\varrho \in C^\infty(\mathbb{R})$ with the following properties
\begin{equation}\label{au1}
\opsupp \varrho \subset [-1,1],\quad |\varrho^{(k)}(x)| \leq C^{k+1}k^{ks}
,\quad \int_{-1}^{+1} \varrho(x)dx = 1,
\end{equation}
for some constant $C > 0$ possibly depending on $s$ but not on $x$ or $k$\footnote{The existence of such a function $\varrho$ is granted by Denjoy-Carleman's Theorem. An explicit construction can be done starting from the function $\exp\left( - \frac{1}{t^\gamma}\right)$, which can be shown to be Gevrey of order $s = 1+1/\gamma$.}. From, \textit{e.g.}, \cite[Theorem 1.6.1]{Rodino}, it is seen that $\mathcal{F}\varrho$ is an entire function on $\mathbb{C}$, which satisfies the estimate on the real line
\begin{equation}\label{eq:fourier_Gevrey}
\exists K,\delta>0,\quad \forall \xi \in \mathbb{R},\quad |\mathcal{F}\varrho(\xi)| \leq K e^{-\delta|\xi|^{1/s}},
\end{equation}
We then set
\begin{equation}\label{au2}
\varrho_\epsilon(x) = \frac{1}{\epsilon} \varrho\left( \frac{x}{\epsilon}\right),
\end{equation}
which, from \eqref{eq:fourier_Gevrey}, satisfies 
\begin{equation}\label{eq:fourier_gevrey_eps}
|\mathcal{F}\varrho_\epsilon (\xi)| \leq K e^{-\delta \epsilon^{1/s}|\xi|^{1/s}},\, \xi\in\mathbb R.
\end{equation}
Note further that since $\varrho_\epsilon$ is supported in $[-\epsilon,\epsilon]$ and square-integrable therein, the Paley--Wiener theorem asserts that its Fourier transform satisfies
\begin{equation}\label{eq:type_eps}
\exists C > 0,\quad \forall z \in \mathbb{C},\quad | \mathcal{F}\varrho_\epsilon (z)| \leq C e^{\epsilon |z|}.
\end{equation}
Now consider $\phi_\epsilon:= (\mathcal{F} \varrho_\epsilon) \phi \in \mathcal{H}(\mathbb{C})$. On the one hand, from \eqref{eq:fourier_gevrey_eps} and being $\phi$ of order $<1$ we get 
\[
|\phi_\epsilon(\xi)| \leq K e^{-\delta \epsilon^{1/s}|\xi|^{1/s}} \times C e^{R|\xi|^\rho},
\]
for some constants $C,R > 0$ and $0 < \rho < 1$, not depending on $\xi \in \mathbb{R}$. Taking $1 < s < 1/\rho$ ensures that the above function is square-integrable over the real line, whence $\phi_\epsilon \in L^2(\mathbb{R})$. On the other hand, because $\phi$ has order $\rho < 1$ and \eqref{eq:type_eps}, it follows that $\phi_\epsilon$ is of exponential type $\leq \epsilon$. Using again the Paley--Wiener theorem we obtain that $\mathcal{F} \phi_\epsilon$ is supported in $[-\epsilon,\epsilon]$, which implies that 
\[
\opsupp \mathcal{F}(\phi_\epsilon f) = \opsupp [(\mathcal{F} \phi_\epsilon) * (\mathcal{F}f)] \subset [-\epsilon , +\epsilon] + \opsupp \mathcal{F}f.
\] 
To conclude the argument by a passage to the limit, we remark that by \eqref{au1} and \eqref{au2}, $\varrho_\epsilon$ is an approximation of the unity, so $\check{\varrho}_\epsilon: x \mapsto \varrho_\epsilon(-x)$ is also an approximation of the unity, whence
\[
\mathcal{F}(\phi_\epsilon f) = \mathcal{F}[(\mathcal{F} \varrho_\epsilon)\phi f ] = \check{\varrho}_\epsilon *\mathcal{F}(\phi f) \xrightarrow[\epsilon \rightarrow 0^+]{L^2(\mathbb{R})} \mathcal{F}(\phi f).
\]
There is thus a sequence $(\epsilon_n)_{n=0}^\infty$ tending to $0^+$ such that $\mathcal{F}(\phi_{\epsilon_n} f)$ goes to $\mathcal{F}(\phi f)$ almost surely, which implies the first inclusion in
\[
\opsupp \mathcal{F}(\phi f) \subset  \limsup_{n \rightarrow \infty} ~ \opsupp \mathcal{F}(\phi_{\epsilon_n} f) \subset \limsup_{n \rightarrow \infty} \left([-\epsilon_n , +\epsilon_n] + \opsupp \mathcal{F}f \right) = \opsupp \mathcal{F}f.
\]
\end{proof}
The given proof constructs a family $(\sigma_\epsilon)_{\epsilon > 0}$ of multipliers for the holomorphic function $\phi$, such that for all $\epsilon > 0$ we have
\begin{equation}\label{eq:multiplier}
    \sigma_\epsilon \in \mathcal{H}(\mathbb{C}) \cap L^2(\mathbb{R}), \quad \sigma_\epsilon \phi \in L^2(\mathbb{R}),\quad \opsupp \mathcal{F}(\sigma_\epsilon \phi) \subset [-\epsilon,\epsilon],
\end{equation}
together with
\begin{equation}\label{eq:multiplier1}
 \forall f \in L^2(\mathbb{R}),\quad \phi f \in L^2(\mathbb{R}) \Longrightarrow \sigma_\epsilon \phi f \xrightarrow[\epsilon \rightarrow 0^+]{L^2(\mathbb{R})} \phi f.
\end{equation}
This allows to give a definition of $\opsupp \mathcal{F}\phi \subset \{0\}$, which is moreover compatible with the standard distributional calculus. It should be noticed that if $\phi$ has no growth restriction on the real axis then its Fourier transform, as a generalized function, may not be supported on the real line. For instance, in \cite[Chapter 2]{gl_eng} it is shown that the Fourier transform of $e^x$ can be identified as the analytic functional which is the evaluation at $z=i$. Moreover, a necessary condition for the existence of $\sigma_\epsilon$ such that \eqref{eq:multiplier} holds is that the logarithmic integral converges, \textit{i.e.}
\begin{equation}\label{eq:log_int}
    \int_\mathbb{R} \frac{\log^+ | \phi(x)|}{1+x^2}dx < \infty,
\end{equation}
see \cite[\S X.A]{K2}. Conversely, if $\phi$ has finite exponential type and satisfies \eqref{eq:log_int}, then from the Beurling--Malliavin Theorem \cite[p. 397]{K2} there exists some $\sigma_\epsilon$ satisfying \eqref{eq:multiplier}, for all $\epsilon >0$. We do not know if one may chose $\sigma_\epsilon$ so as \eqref{eq:multiplier1} is also satisfied.
\subsection{A Plancherel result in Gevrey classes}
In this subsection, we show Theorem \ref{theo:plancherel_gevrey}.  Let $f \in \mathcal{S}'(\mathbb{R})$ be such that $\mathcal{F}f \in L^1_{\oploc}(\mathbb{R})$. Let $s ,R > 0$ and $\gamma \in \mathbb{R}$. The proof is done by estimating the quantity
\[
\int_\mathbb{R} |\mathcal{F}f(\xi)\omega(\xi)|^2d\xi,\quad \omega(\xi):= (1+|\xi|)^\gamma e^{R |\xi|^{1/s}},
\]
whether it is finite or not. We will replace the weight $\omega$ by another weight $\varpi$ of the form 
\[
\varpi(z) = \sum_{k=0}^\infty a_k z^{2k}, \quad a_k \geq 0,
\] 
which is such that $\omega \asymp \varpi$ on $\mathbb{R}$ and $\varpi$ is an entire and even function on $\mathbb{C}$. Here and throughout, $\asymp$ means that we both have $\lesssim$ and $\gtrsim$. Let us assume provisionally that such a weight $\varpi$ exists, we then compute
\begin{align*}
\int_\mathbb{R} |\mathcal{F}f(\xi)\omega(\xi)|^2d\xi &\asymp
\int_\mathbb{R} | \mathcal{F}f(\xi) \varpi(\xi)|^2 d\xi \\
&= \int_\mathbb{R} \left| {\mathcal{F}f(\xi)}\sum_{k=0}^\infty a_k\xi^{2k} \right|^2 d\xi\\
&=\sum_{k,l=0}^{\infty}\int_{\mathbb R}\left| \mathcal{F}f(\xi)\right|^2 \xi^{2k}a_k\xi^{2l}a_l  d\xi \\
&=\sum_{k,l=0}^{\infty}a_ka_l \int_{\mathbb R}| {\mathcal{F}f^{(k+l)}(\xi)}|^2   d\xi\\&
\asymp \sum_{k,l=0}^{\infty}a_ka_l ||f^{(k+l)}||_{L^2(\mathbb R)}^2 \\
&=\sum_{n=0}^{\infty} ||{f^{(n)}}||_{L^2(\mathbb R)}^2 \sum_{k=0}^n a_ka_{n-k}. \label{eq:rhs_plancherel_gevrey} \numberthis{}
\end{align*}
We are thus left to find such a $\varpi$ and to estimate the quantity 
\[
A_n:=  \sum_{k=0}^n a_ka_{n-k},
\]
as $n \rightarrow \infty$. For $\alpha,\beta > 0$ we introduce the two-parameters Mittag-Leffler function 
\[
E_{\alpha,\beta}(z):= \sum_{k=0}^\infty \frac{z^k}{\Gamma(\alpha k + \beta)},
\]
which is entire on $\mathbb{C}$. It admits the following asymptotic expansion on the positive real axis: 
\[
E_{\alpha,\beta}(x) \sim \frac{1}{\alpha} x^{\frac{1-\beta}{\alpha}} \exp x^{1/\alpha},\quad x \rightarrow +\infty,
\]
see \cite[Theorems 4.3 and 4.4]{mittag_leffler}. 

\noindent\textbf{A particular case.} Let us assume that $\gamma < 0$, $R = 1$, and put 
\begin{equation}\label{eq:def_param_mittag_leffler}
    \varpi(z) := E_{\alpha,\beta + 1}(z^2) = \sum_{k=0}^\infty \frac{z^{2k}}{\Gamma(\alpha k + \beta + 1)}, \quad \frac{1}{s} = \frac{2}{\alpha},\quad \gamma = -2 \frac{\beta}{\alpha},  
\end{equation}
so that $\varpi$ is entire and even on $\mathbb{C}$, bounded below by a positive constant on $\mathbb{R}$, and admits the asymptotic expansion 
\[
\varpi(\xi) \sim \frac{1}{\alpha} |\xi|^\gamma \exp |\xi|^{1/s},\quad \xi \rightarrow \pm \infty. 
\]
To estimate $A_n$ we start by writing Stirling's formula as 
\[
\frac{1}{a_k} = \Gamma(\alpha k + \beta + 1) \asymp (1+k)^{\beta + 1/2} \left( \frac{\alpha k}{e} \right)^{\alpha k},\quad  \forall k\in \mathbb N.
\]
Note that we write $(1+k)^{\beta + 1/2} $ rather than $k^{\beta + 1/2}$, so that the right-hand side of the above estimate never vanishes. We obtain, for $0 \leq k \leq n$ and $n \geq 1$,
\begin{align*}
\frac{1}{a_ka_{n-k}} &\asymp \left[ (1+k)(1+n-k) \right]^{\beta + 1/2} \left( \frac{\alpha k}{e} \right)^{\alpha k} \left( \frac{\alpha (n-k)}{e} \right)^{\alpha (n-k)} \\
&= n^{1 + 2\beta} \left[ \left( \frac{1}{n}+\frac{k}{n}\right)\left( \frac{1}{n}+1-\frac{k}{n}\right) \right]^{\beta + 1/2} \left( \frac{n\alpha}{e}\right)^{\alpha n} \left[ \left(\frac{k}{n} \right)^{\alpha k / n}  \left(1-\frac{k}{n}\right)^{\alpha \left(1-\frac{k}{n}\right)} \right]^n \\
&= \left( \frac{n\alpha}{e}\right)^{\alpha n} n^{1 + 2\beta} g_n(k/n)^{-1}  h(k/n)^n,
\end{align*}
with 
\[
g_n(x) = \left[ \left( \frac{1}{n}+x \right)\left( \frac{1}{n}+1-x \right) \right]^{-\beta - 1/2},\quad h(x) = x^{\alpha x} (1-x)^{\alpha(1-x)}.
\]
This brings 
\begin{equation}\label{eq:asymp_An}
A_n \asymp \left( \frac{e}{\alpha n} \right)^{\alpha n} \frac{1}{n^{2 \beta }} \cdot \frac{1}{n} \sum_{k=0}^n g_n(k/n) h(k/n)^{-n} , \quad (n \geq 1),
\end{equation}
which reduces the estimation of $A_n$ to that of the second factor on the right-hand side of \eqref{eq:asymp_An}. We write 
\begin{equation}\label{eq:decomp_laplace_discrete}
    \frac{1}{n}\sum_{k=0}^n g_n(k/n) h(k/n)^{-n} =  g_n(1/2) \frac{1}{n} \sum_{k=0}^n h(k/n)^{-n} + \frac{1}{n}\sum_{k=0}^n (g_n(k/n)-g_n(1/2)) h(k/n)^{-n}.
\end{equation}
To estimate the above right-hand side, we begin by estimating the first of the two sums. We write 
\[
    \frac{1}{n}\sum_{k=0}^n h(k/n)^{-n} = \frac{1}{n} \sum_{k=0}^n e^{-n\log h(k/n)},
\]
which looks like a discrete analog of the integral 
\[
\int_0^1 e^{-n \log h(x)}dx.
\]
For the above integral, the asymptotic behaviour is well known by Laplace's method. Following \cite{laplace_discrete}, we rely on the following Lemma.
\begin{Lemma}\label{lem:laplace_discrete}
Let $u: [0,1] \rightarrow \mathbb{R}$ be in $C^2(0,1) \cap C[0,1]$ with global strict minimum at $x_0 \in (0,1)$, for which $u''(x_0) > 0$. Then 
\[
\frac{1}{n} \sum_{k=0}^n e^{-nu(k/n)} \sim \sqrt{\frac{2 \pi}{u''(x_0) n}} e^{-nu(x_0)},\quad n \rightarrow \infty.
\]
\end{Lemma}
The proof of this Lemma is elementary but rather lengthy, we defer it to the Appendix \ref{sec:appB}. Applying this result with $u = \log h$ we obtain the following asymptotic behaviour, for which we recall that $h(1/2) = 2^{-\alpha}$,
\[
\frac{1}{n} \sum_{k=0}^n h(k/n)^{-n} \asymp \frac{2^{\alpha n}}{\sqrt{n}}, \quad (n \geq 1).
\]
Because $g_n(1/2) \asymp 1$, we deduce 
\begin{equation}\label{eq:asymptotics_laplace_discrete}
    g_n(1/2)\frac{1}{n}\sum_{k=0}^n h(k/n)^{-n} \asymp \frac{2^{\alpha n}}{\sqrt{n}},\quad (n \geq 1).
\end{equation}
Coming back to \eqref{eq:decomp_laplace_discrete}, we see that \eqref{eq:asymptotics_laplace_discrete} may be plugged-in \eqref{eq:asymp_An} as soon as
\begin{equation}\label{eq:difference_negligible}
    \frac{1}{n} \sum_{k=0}^n (g_n(k/n)-g_n(1/2)) h(k/n)^{-n} = o\left(  \frac{2^{\alpha n}}{\sqrt{n}} \right),\quad n \rightarrow \infty. 
\end{equation}
To show that \eqref{eq:difference_negligible} holds, we introduce a sequence $(\epsilon_n)_{n\in\mathbb N^*}$ with $0<\epsilon_n<1/2$ for any $n\in\mathbb N^*$, to be specified later on, and compute
\begin{align*}
    \left| \frac{1}{n} \sum_{k=0}^n (g_n(k/n)-g_n(1/2)) h(k/n)^{-n} \right| &\leq \frac{1}{n} \sum_{\left| \frac{k}{n} - \frac{1}{2} \right| \leq \epsilon_n} |g_n(k/n)-g_n(1/2)| h(k/n)^{-n}  \\
    & \quad \quad + \frac{1}{n} \sum_{\left| \frac{k}{n} - \frac{1}{2} \right| > \epsilon_n} |g_n(k/n)-g_n(1/2)| h(k/n)^{-n}. \label{eq:rhs_laplace_discrete} \numberthis{}
\end{align*}
We estimate the two terms of the right-hand side of \eqref{eq:rhs_laplace_discrete} separately. For the second one, we observe that $g_n$ is positive on $[0,1]$, where it reaches its maximal value at $x= 0$ and $x=1$, for which 
\[
g_n(0) \sim n^{\beta +1/2},\quad n \rightarrow \infty.
\]
Moreover, 
\[
\min_{|x-x_0| \geq \epsilon_n}h(x) = h(1/2+\epsilon_n),
\]
hence 
\[
\frac{1}{n} \sum_{\left| \frac{k}{n} - \frac{1}{2} \right| > \epsilon_n} |g_n(k/n)-g_n(1/2)| h(k/n)^{-n} \lesssim n^{\beta+1/2}h(1/2+\epsilon_n)^{-n},\quad n\geq 1.
\]
Recalling $h(1/2) = 2^{-\alpha}$, we deduce
\[
n^{\beta+1/2}h(1/2 + \epsilon_n)^{-n} = o\left( \frac{2^{\alpha n}}{\sqrt{n}} \right) \Longleftrightarrow n\left[ \log h(1/2+\epsilon_n) - \log h(1/2) \right] - (\beta+1) \log n \rightarrow + \infty,
\]
and by log convexity of $h$, there exists some constant $\kappa > 0$ such that 
\[
\log h(1/2+\epsilon_n) - \log h(1/2) \geq \kappa \epsilon_n^2.
\]
We take 
\[
\epsilon_n = \frac{1}{n^\mu},\quad 0 < \mu < 1/2,
\]
so that the second term on the right-hand side of \eqref{eq:rhs_laplace_discrete} has the appropriate size. For the first term, we write
\[
g_n(x) = \frac{1}{i_n(x)^{2\beta + 1}},\quad i_n(x) = \left( \frac{1}{n}+x \right)\left( \frac{1}{n}+1-x \right),
\]
and compute for $|x-1/2| \leq \delta$
\[
|i_n(x) - 1/4| \leq |i_n(x) - i_n(1/2)| + |i_n(1/2) - 1/4| \leq \delta^2 + \frac{2}{n}.  
\]
Therefore, for small $\delta$ and large $n$, the term $i_n(x)$ is close to $1/4$, around which the function $x \mapsto x^{-2\beta-1}$ is Lipschitz. We thus put $\delta = \epsilon_n$, so that for $n$ large enough
\[
\frac{1}{n} \sum_{\left| \frac{k}{n} - \frac{1}{2} \right| \leq \epsilon_n} |g_n(k/n)-g_n(1/2)| h(k/n)^{-n} \lesssim \left( \epsilon_n^2 + \frac{1}{n}\right) 2^{\alpha n}.
\]
We take $1/4 < \mu < 1/2$ in the definition of $\epsilon_n$ so that the above contribution is negligible compared with $2^{\alpha n}/ \sqrt{n}$ as $n \rightarrow \infty$. This shows that \eqref{eq:difference_negligible} holds. 

We may thus plug \eqref{eq:asymptotics_laplace_discrete} in \eqref{eq:asymp_An}, which yields 
\[
A_n \asymp \left( \frac{2e}{\alpha n} \right)^{\alpha n} \frac{1}{n^{2 \beta + 1/2}}, \quad (n \geq 1).
\]
From \eqref{eq:rhs_plancherel_gevrey} and \eqref{eq:def_G_sRgamma_Mn}, we see that the sequence $(M_n)$ should satisfy $M_n \asymp 1/\sqrt{A_n}$, and using \eqref{eq:def_param_mittag_leffler} we compute
\[
\frac{1}{\sqrt{A_n}} \asymp \left( \frac{ns}{e} \right)^{ns}n^{-s\gamma + 1/4},\quad (n\geq 1).
\]
This ends the proof of Theorem \ref{theo:plancherel_gevrey}, under the additional assumptions $R = 1$ and $\gamma < 0$. \newline
\newline 
\textbf{General case.} We extend the result to the cases of $R > 0$ and $\gamma \in \mathbb{R}$ arbitrary by acting on $f$. The case of general $R > 0$ is easily deduced from the case $R = 1$ by replacing $f(x)$ by $f(R^sx)$, so, without loss of generality, we may assume that $R = 1$. To allow $\gamma$ to be non-negative, we pick $p \in \mathbb{N}$ large enough so that $\tilde{\gamma}:= \gamma -p<0$. We then observe that for every $f \in \mathcal{S}'(\mathbb{R})$ with $\mathcal{F}f \in L^1_{\oploc}(\mathbb{R})$, there holds 
\begin{align*}
    \int_\mathbb{R} | \mathcal{F}f(\xi) e^{-|\xi|^{1/s}} (1+|\xi|)^\gamma|^2 d\xi &\asymp \sum_{k=0}^p \int_\mathbb{R} | \mathcal{F}f^{(k)}(\xi) e^{-|\xi|^{1/s}} (1+|\xi|)^{\tilde{\gamma}} |^2 d\xi \\
    &\asymp \sum_{k=0}^p \| f^{(k)} \|_{\mathcal{G}_{s,1,\tilde{\gamma}}}^2 \\
    &= \sum_{k=0}^p \sum_{n=0}^\infty \left( \frac{\|f^{(k+n)}\|_{L^2(\mathbb{R})}}{M_n} \right)^2 \\
    &\asymp \sum_{n=0}^\infty \left( \frac{\|f^{(n)}\|_{L^2(\mathbb{R})}}{N_n} \right)^2,
\end{align*}
where 
\[
M_n = \left( \frac{ns}{e} \right)^{ns} (1+n)^{-s\tilde{\gamma}+1/4},
\]
and $(N_n)$ is any positive sequence with 
\[
N_n \asymp M_{n-p},\quad n \geq p.
\]
For instance, we can take
\[
N_n:= \left( \frac{ns}{e} \right)^{ns} (1+n)^{-s\gamma+1/4}.
\]
This shows that 
\[
\int_\mathbb{R} | \mathcal{F}f(\xi) e^{-|\xi|^{1/s}} (1+|\xi|)^\gamma|^2 d\xi \asymp \|f\|_{\mathcal{G}_{s,1,\gamma}},
\]
which ends the proof of Theorem \ref{theo:plancherel_gevrey}.
\subsection{Proof of the main results}\label{smt}
For the system \ref{eq:heat} we consider transposition solutions \cite[\S 2.3]{Coron}. 
\begin{Definition}
Let $u \in L^2(0,\infty)$, a solution of \eqref{eq:heat} is $z \in C([0,\infty);L^2(0,1))$ such that for all $0 < T < \infty$ and all $\varphi^T \in L^2(0,1)$, denoting $\varphi(t,x)$ the solution of 
\[
\left\lbrace \begin{array}{rclcc}
    -\varphi_t(t,x) & = & \varphi_{xx}(t,x),& 0 < t < T,& 0 < x < 1,  \\
    \varphi_x(t,1) & = & 0, \\
    \varphi_x(t,0) &=& 0,\\
    \varphi(T,x) &=& \varphi^T(x),
\end{array}\right.
\]
we have 
\[
\int_0^1 z(T,x)\varphi^T(x)dx = \int_0^T \varphi(t,1) u(t)dt.
\]
\end{Definition}
It is easy to see that \eqref{eq:heat} is then well-posed in the sense that for all $u \in L^2(0,\infty)$, there exists a unique solution $z$ to \eqref{eq:heat}, and that 
\[
\forall 0 < T < \infty,\quad \exists C > 0,\quad \forall u \in L^2(0,\infty),\quad \| z \|_{C([0,T] ; L^2(0,1))} \leq C \|u\|_{L^2(0,T)}.
\]
Moreover, introducing the spectral decomposition of the Neumann Laplacian 
\[
\lambda_j = (j \pi)^2,\quad e_j(x) = \left\lbrace \begin{array}{cc}
    1 & j = 0, \\
    \sqrt{2} \cos(j\pi x) & j \geq 1,
\end{array}\right. ,\quad (j \in \mathbb{N})
\]
one easily verifies that $z$ assumes the representation
\begin{equation}\label{eq:repr_z}
    z(t) = \int_0^t \sum_{j=0}^\infty e^{-\lambda_j (t-\sigma)} e_j(1)u(\sigma)e_j d\sigma,    
\end{equation}
where for fixed $t$ there holds 
\[
\int_0^t \sum_{j=0}^\infty \left\| e^{-\lambda_j (t-\sigma)} e_j(1) u(\sigma)e_j \right\|_{C[0,1]} d\sigma < \infty.
\]
Thus, the function $z$ belongs to $C([0,\infty) \times [0,1])$, and the output $y(t)$ makes sense as the pointwise evaluation of $z(t,x)$ at $x=0$. In fact, using the hypoellipticity of the heat operator $\partial_t - \partial_{xx}$, one can further show that $y \in C^\infty_{(0)}[0,\infty)$.
\begin{proof}[Proof of  Theorem \ref{theo:charac_Y_Tinfty}]
In what follows all the signals defined on $(0,\infty)$ are extended by 0 on $(-\infty,0)$. Assume first that $y \in \mathcal{Y}(0,\infty)$, let us show that $y$ is Laplace transformable and its Laplace transform satisfies \eqref{eq:laplace_rewrite}. A formal proof is done by taking the Laplace transform in time in \eqref{eq:heat} (without checking whether it is legitimate or not), solving the resulting second order differential equation in space, using the boundary conditions given by the second and third lines of \eqref{eq:heat} and finally evaluating at $x=0$. 

The above method will be applied in \S\ref{sec:Neu_Neu} to another system. For \eqref{eq:heat} it will be more convenient to work with the representation formula \eqref{eq:repr_z}. The representation formula \eqref{eq:repr_z} may be evaluated at $x=0$, yielding 
\[
y(t) = z(t,0) = \int_0^t k(t-\sigma)u(\sigma)d\sigma,\quad k(t) := 1 + 2 \sum_{j=1}^\infty e^{-(j\pi)^2t}(-1)^j,
\]
where the defined kernel $k$ is extended by $0$ for $t < 0$. The function $k$ may be rewritten using the Poisson summation formula as 
\[
k(t) = \frac{1}{\sqrt{\pi t}} \sum_{m=-\infty}^{+\infty} e^{- \frac{(m+1/2)^2}{t}}.
\]
Combining the two representations of $k$, we arrive to $k \in C^\infty_{(0)}[0,\infty) \cap L^\infty(0,\infty)$. The last fact can be used as an alternative proof of $y \in C_{(0)}^\infty[0,\infty)$. The signals $k,u,y$ are therefore all Laplace transformable, with abscissa of absolute convergence no greater than $0$. The convolution rule further applies: 
\[
\hat{y}(s) =\hat{k}(s) \hat u(s),\quad s \in \mathbb{C}_+,
\]
where an explicit computation gives that
$$\hat{k}(s) =\sum_{m=0}^\infty   \mathcal{L}\left(\frac{1}{\sqrt{\pi t}}  e^{- \frac{(m+1/2)^2}{t}}\right)(s) = \frac{1}{\sqrt{s}}\sum_{m=0}^\infty     e^{- 2|m+1/2|\sqrt{s}}=\frac{1}{\sqrt{s}}\frac{2e^{\sqrt{s}}}{e^{2\sqrt{s}}-1}.$$
In the above computation, the interversion series-integral is justified by considering first $s > 0$, which yields only positive quantities in the summand. We deduce that the transfer function of \eqref{eq:heat} is given by 
$$\hat{k}(s)=\frac{1}{\sqrt{s}\sinh \sqrt{s}},$$
whence
\[
\hat{y}(s) =\frac{\hat{u}(s)}{\sqrt{s}\sinh \sqrt{s}}, 
\]
which rewrites as
\begin{equation}\label{luy}
\hat{u}(s) = \sqrt{s}\sinh \sqrt{s} \hat{y}(s) = \frac{\sinh
\sqrt{s}}{\sqrt{s}} \hat{\dot{y}}(s) = \phi(s) \hat{\dot{y}}(s),
\end{equation}
where we have denoted 
\begin{equation}\label{defi}
\phi(s):= \frac{\sinh \sqrt{s}}{\sqrt{s}}.
\end{equation}
Note that $1/\phi$ is bounded on $\mathbb{C}_+$, so that $\hat{\dot{y}} \in \mathcal{H}^2(\mathbb{C}_+)$ and $\dot{y} \in L^2(0,\infty)$. In particular, the Laplace transform $\hat{\dot{y}}$ has the non tangential boundary value $\hat{\dot{y}}(i\xi) = \mathcal{F}\dot{y}(\xi)$. Now, coming back to 
\[
\phi \hat{\dot{y}} = \hat{u} \in \mathcal{H}^2(\mathbb{C}_+),
\]
we take advantage of the above function having a square-integrable boundary values to deduce that $\dot{y}$ satisfies \eqref{eq:charac_Y_fourier}. From the Paley--Wiener Theorem \ref{theo:paley_wiener_isom} and Proposition \ref{prop:hardy}, the previous reasoning can be reversed, so that 
\[
\mathcal{Y}(0,\infty) = \left\lbrace y \in \mathcal{S}'(\mathbb{R}): \opsupp y \subset [0,\infty),\quad \dot{y} \mbox{ satisfies }\eqref{eq:charac_Y_fourier} \right\rbrace. 
\]
The rest of the proof of Theorem \ref{theo:charac_Y_Tinfty} then follows from the Gevrey-Plancherel Theorem \ref{theo:plancherel_gevrey}.
\end{proof}
\begin{proof}[Proof of Theorem \ref{theo:Y_finite_T}]
It is enough to show that any function $\varphi \in C^\infty_{(0)}[0,T]$ satisfying the two conditions \eqref{eq:reg_Y_T_fini} and \eqref{eq:final_in_R} lies in $\mathcal{Y}(0,T)$. For, let $\varphi$ be as such, from \eqref{eq:final_in_R} there is a control $u \in L^2(0,T)$ such that 
\[
z(T,x) = \sum_{k=0}^\infty \varphi^{(k)}(T^-) \frac{x^{2k}}{(2k)!},\quad 0 < x < 1,
\]
where $z$ is the solution of \eqref{eq:heat} with control $u$. In particular, we have $\partial_t^k z(T,0) =  \varphi^{(k)}(T^-)$ for every $k \in \mathbb{N}$. We then extend the control $u$ by 0 on $(T,\infty)$, which extends $z$ for times $t > T$ as a $C^\infty([0,\infty)\times (-1,1))$. Consider then 
\[
\psi(t):= \left\lbrace \begin{array}{rl}
    \varphi(t), & 0 < t < T, \\
    z(t,0), & T < t < \infty.
\end{array}\right.
\]
By Theorem \ref{theo:charac_Y_Tinfty}, the function $\psi$ lies in $\mathcal{Y}(0,\infty)$. Therefore, its restriction to $(0,T)$ lies in $\mathcal{Y}(0,T)$. This ends the proof of Theorem \ref{theo:Y_finite_T}.
\end{proof}
\begin{Remark}\label{rem:no_adm_C}
For an initial data $z^0 \in L^2(0,1)$, the solution $z$ of 
\begin{equation}\label{eq:homo}
    \left\lbrace \begin{array}{rclcc}
z_t(t,x) &=& z_{xx}(t,x),& 0 < x < 1,&t>0,\\
z_x(t,1) &=& 0, \\
z_x(t,0) &=& 0,\\
z(0,x) &=& z^0(x),
\end{array}\right.
\end{equation}
is of class $C^\infty((0,\infty) \times (-1,1))$ and the output $y(t) = z(t,0)$ is in $L^p(0,T)$ for all $1 \leq p < 4$ and $0 < T < \infty$. The value of $p$ cannot be increased to $p > 4$. This however has no impact on the proof of Theorem \ref{theo:Y_finite_T}, as therein we only consider such outputs for initial data in the reachable space, which are seen to be exactly shifted outputs of the system \eqref{eq:heat} (with a control and zero initial data). In other words, we do not need the existence of the output $y$ for \eqref{eq:homo}.
\end{Remark}
\begin{proof}[Proof of Corollary \ref{coro:loc_tracable}]
Let $\varphi \in C_{(0)}^\infty[0,T]$ satisfying \eqref{eq:reg_Y_T_fini} and $0 < \delta < T$. Introduce $\chi \in C^\infty[0,T]$ such that 
\[
|\chi^{(n)}(t)| \leq C^{n+1}n^{ns},\quad \chi(t) = 1~~(0 \leq t \leq T-\delta),\quad \chi(t) = 0~~(T-\delta/2 \leq t \leq T),
\]
where $C > 0$ and $1 < s < 2$ are independent of $n \in \mathbb{N}$ and $t \in [0,T]$. Reasoning similarly to \cite[Lemma 3.7]{reachable_martin} we see that $\chi \varphi$ satisfies again \eqref{eq:reg_Y_T_fini}. Moreover, since $\chi \varphi$ is flat at $t=T$ it also satisfies \eqref{eq:final_in_R}. Thus from Theorem \ref{theo:Y_finite_T} we have $\chi \varphi \in \mathcal{Y}(0,T)$. Therefore, the restriction of $\chi \varphi$ to $[0,T-\delta]$ belongs to $\mathcal{Y}(0,T-\delta)$, which ends the proof. 
\end{proof}
\subsection{Proof of Proposition \ref{prop:borel_LR}}
In this subsection we prove Proposition \ref{prop:borel_LR}. We will rely on the following Lemma taken from \cite[Theorem 3.8]{reachable_martin}, which is proved using the Borel summation \cite{ramis}.
\begin{Lemma}\label{lem:borel_sum}
Let $(a_n)$ be a sequence such that the power series  
\[
f(z):= \sum_{n=0}^\infty \frac{a_{n+1}}{n!(n+1)!}z^n,
\]
has a nonzero radius of convergence. Assume that $f$ extends to a holomorphic function on a neighborhood of $\mathbb{R}_-$ in $\mathbb{C}$ with,
\begin{equation}\label{eq:growth_cut}
\exists C > 0,\quad \forall n \geq 0,\quad \forall x < 0,\quad |f^{(n)}(x)| \leq C |f^{(n)}(0)|.     
\end{equation}
Then there exists $\varphi \in C^\infty[-1,0]$ such that
\[
\varphi^{(n)}(0) = a_n,\quad |\varphi^{(n)}(t)| \lesssim |a_n|,\quad (n \geq 0,\quad -1 \leq t \leq 0).
\]
\end{Lemma}
\begin{proof}[Proof of Proposition \ref{prop:borel_LR}]
Following \cite[Remark 3.3, item 3]{reachable_martin}, we consider the function
\[
f(z):= \frac{1}{1-z},
\]
which satisfies the hypotheses of Lemma \ref{lem:borel_sum}. Hence  there exists $\phi \in C^\infty[-1,0]$ such that  
\[
\phi^{(n)}(0) = b_n,\quad |\phi^{(n)}(t)| \lesssim  b_n,\quad (n \geq 0,\quad -1 \leq t \leq 0),
\]
where $b_{n+1} := n!(n+1)!$ and $b_0 = 1$. By Stirling's formula, 
\[
b_n \asymp \frac{(2n)!}{\sqrt{1+n}4^n}.
\]
We then consider $a_n:= 4^nb_n$. By scaling, there exists $\varphi \in C^\infty[-1,0]$ such that  
\[
\varphi^{(n)}(0) = a_n,\quad |\varphi^{(n)}(t)| \lesssim \frac{(2n)!}{\sqrt{1+n}},\quad (n \geq 0,\quad -1 \leq t \leq 0).
\]
Assume by contradiction that there exists some $\psi \in C^\infty[0,1]$ such that \eqref{eq:reg_Y_T_fini} holds with $T = 1$. The sequence $(a_n)$ is then the sequence of successive derivatives at an interior point of a $\mathcal{G}_{2,1/\sqrt{2},-1/2}$ function. Hence, by Corollary \ref{coro:loc_tracable} we deduce
\[
 \sum_{n=0}^\infty a_n \frac{\zeta^{2n}}{(2n)!} \in \mathcal{R}.
\]
However, a sharper estimate of $b_n$ yields
\[
\frac{a_n}{(2n)!} = \sqrt{\frac{\pi}{n}} +  O \left( \frac{1}{n^{3/2}} \right),
\]
where, for any bounded sequence $(c_n)$, the series $\sum c_n \zeta^{2n} / n^{3/2}$ converges absolutely in $A^2(D(0,1))$. We thus deduce that the function
\[
\opLi_{-1/2}(\zeta^2) = \sum_{n=1}^\infty \frac{\zeta^{2n}}{n^{-1/2}},
\]
belongs to $A^2(\Omega)$. For every $s \in \mathbb{R}$, the function 
\[
\opLi_s(\zeta) = \sum_{n=1}^\infty \frac{\zeta^{n}}{n^{s}},
\]
is a special function often called the polylogarithm, or Jonqui\`ere's function, and is a special case of Lerch's function. It is defined by a power series with radius of convergence 1 and has the following asymptotic behaviour
\begin{equation}\label{eq:asymptotics_polylog}
    \opLi_s(\zeta) = \frac{\Gamma(1-s)2^{s-1}}{(1-\zeta)^{1-s}} + O(1), \quad \zeta \rightarrow 1,\quad \zeta \in D(1/2,1/2),
\end{equation}
for every $s \neq 1,2,3,...$, see \cite[\S 1.11, equation (8)]{transcendental}. It is then clear that $\opLi_{-1/2}(\zeta^2)$ cannot lie in $A^2(\Omega)$, which is the sought contradiction.  
\end{proof}
We end this Section with two consequences of Proposition \ref{prop:borel_LR}. First, we formulate below the extension problem for Gevrey functions of order $2$, to be compared with Definition \ref{def:carl}.
\begin{Definition}\label{def:ext}
We say that the extension problem (for $s=2$) is solvable with loss $\rho \in (0,1)$ if: for every functions $\phi \in C^\infty[-1,1]$ such that 
\[
\exists C,R_0 > 0,\quad \forall n \in \mathbb{N},\quad \forall t \in [-1,1],\quad |\phi^{(n)}(t)| \leq C \frac{(2n)!}{R_0^{2n}},
\]
and for every $\epsilon > 0$, there exists a function $\varphi \in C^\infty[-2,2]$ such that 
\[
\forall t \in [-1,1],\quad \phi(t) = \varphi(t),
\]
and 
\[
\exists C' > 0, \quad \forall n \in \mathbb{N},\quad \forall t \in [-2,2],\quad |\varphi^{(n)}(t)| \leq C' \frac{(2n)!}{(\rho R_0 - \epsilon)^{2n}}.
\]
\end{Definition}
It is clear that if the extension problem can be solved with some loss factor $\rho$, then it can also be solved with any $\tilde\rho \leq \rho$. This allows to define $\rho^{Ext}$ as the optimal (\textit{i.e.} maximal) loss factor $\rho$ such that the extension problem is solvable with loss $\rho$. Recall that from Mitjagin's theorem the interpolation problem can be solved for Gevrey functions of order $2$, the optimal loss factor being $\rho^{Carl} = 1/\sqrt{2}$. Proposition \ref{prop:borel_LR} shows that $\rho^{Ext} \leq 1/\sqrt{2}$ and it easy to see that $\rho^{Ext} \geq \rho^{Carl}$, hence $\rho^{Ext} = \rho^{Carl}$. In other words: there is no way to extend functions in the Gevrey class of order $2$ which yields a better loss factor than using the solution of the interpolation problem.\newline
\newline
Secondly, we point out that Proposition \ref{prop:borel_gevrey} describes the range of the Borel operator, over some Hilbert spaces of Gevrey-2 functions on some interval $I$, when the point $t_0$ at which the derivatives are taken is an interior point of $I$. Such description is harder when $t_0$ is a boundary point of $I$, in fact we have the following.
\begin{Proposition}
Let $R > 0$ and denote $\mathcal{B}_{x_0}^R$ the Borel operator acting on $G^{2,R}[-1,1]$, at the point $x_0 \in [-1,1]$. We then have:
\begin{itemize}
    \item For every $x_0 \in (-1,1)$, $\opRange \mathcal{B}_{x_0}^R = \opRange \mathcal{B}_0^R \subset \opRange \mathcal{B}_{-1}^R \cap \opRange \mathcal{B}_1^R$.
    \item For every $\epsilon > 0$, there holds $\opRange \mathcal{B}_{\pm 1}^R \subset \opRange \mathcal{B}_{\mp 1}^{R/\sqrt{2} - \epsilon}$.
    \item For every $\epsilon > 0$, we do not have  $\opRange \mathcal{B}_{\pm 1}^R \subset \opRange \mathcal{B}_{\mp 1}^{R/\sqrt{2} + \epsilon}$.
\end{itemize}
\end{Proposition}
\begin{proof}
The first item is trivial, the second one is by Mitjagin's Theorem \ref{theo:mitjagin1}, and the third item is a consequence of Proposition \ref{prop:borel_LR}.
\end{proof}
\section{Similar systems}
\subsection{Other boundary conditions}
\subsubsection{Neumann-to-Neumann}\label{sec:Neu_Neu}
Let us consider the system 
\begin{equation}\label{eq:heat_NN}
\left\lbrace \begin{array}{rclcc}
z_t(t,x) &=& z_{xx}(t,x),& 0 < x < 1,&t>0,\\
z_x(t,1) &=& u(t), \\
z(t,0) &=& 0,\\
z(0,x) &=& z^0(x),  \\
y(t) &=& z_x(t,0),
\end{array}\right.
\end{equation}
which we call ``Neumann-to-Neumann'' because the control is through the Neumann action and the output is the Neumann trace. Let us first disregard the output $y$; similarly as for \eqref{eq:heat} we consider transposition solutions and it is easy to check that \eqref{eq:heat_NN} is well-posed on the state space $L^2(0,1)$. When $z^0 = 0$ and $u \in L^2(0,\infty)$, extending the solution $z$ of \eqref{eq:heat_NN} as an odd function of $x \in (-1,1)$ and using the hypoellipticity of the heat equation one sees that $z \in C^\infty([0,\infty) \times (-1,1))$. Hence, when $z^0 = 0$, the output $y(t) = z_x(t,0)$ is just a pointwise evaluation. When $u = 0$ and $z^0 \in L^2(0,1)$, one sees that $z \in C^\infty((0,\infty) \times (-1,1))$ and $z_x(\cdot,0) \in L^1(0,\infty)$, but we will not use this fact. 

Let now $u \in L^2(0,\infty)$, we compute the Laplace transform of $y$. A representation formula similar to \eqref{eq:repr_z} would yield $z \in C([0,\infty) \times [0,1])$ but be of no use here because the output is the Neumann trace, rather we shall pass the equation \eqref{eq:heat_NN} to the Laplace transform. To this aim, we first take $u \in H^1(0,\infty)$ such that $u(0) = 0$. We consider the function 
\[
\xi(t,x) = z(t,x) - x(x-1)u(t),
\]
which is continuous in both variables and solves 
\[
\left\lbrace \begin{array}{rclcc}
\xi_t(t,x) &=& \xi_{xx}(t,x) + f(t,x),& 0 < x < 1,&t>0,\\
\xi_x(t,1) &=& 0, \\
\xi(t,0) &=& 0,\\
\xi(0,x) &=& 0, 
\end{array}\right., \quad f(t,x) = -x(x-1)\dot{u}(t) + \frac{u(t)}{2}.
\]
The source term $f$ is of class $L^2(0,\infty ; L^2(0,1))$, hence by maximal regularity \cite[Proposition 3.7, \S II.1]{bensoussan_control} we have 
\[
\xi \in H^1(0,\infty; L^2(0,1)) \cap L^2(0,\infty;H^2(0,1)).
\]
The function $z$ is therefore also in the above class and we are allowed to pass \eqref{eq:heat_NN} to the Laplace transform with respect to $t$: for all $s \in \mathbb{C}_+$, the function $\hat{z}(s,\cdot)$ satisfies 
\[
\left\lbrace \begin{array}{rcl c}
    s \hat{z}(s,x) & = & \partial_{xx} \hat{z}(s,x) ,& \mbox{a.e. } x \in (0,1),  \\
    \partial_x \hat{z}(s,1) &=& \hat{u}(s),\\
    \hat{z}(s,0) &=& 0,
\end{array}\right.
\]
from which we deduce that 
\[
\hat{z}(s,x) = \frac{\hat{u}(s) \sinh(\sqrt{s}x)}{\sqrt{s}\cosh(\sqrt{s})}.
\]
We then evaluate 
\[
\hat{y}(s) = \partial_x\hat{z}(s,0) = \frac{\hat{u}(s)}{\cosh\sqrt{s}},
\]
so that the transfer function of \eqref{eq:heat_NN} is thus 
\[
\mathbf{H}_{\opNeu - \opNeu}(s) = \frac{1}{\cosh \sqrt{s}}.
\]
Notice that the above function is bounded in modulus over $\mathbb{C}_+$, from which one deduces by approximation arguments that for general $u \in L^2(0,\infty)$, the output $y$ is $L^2(0,\infty)$. 

Repeating the analysis proposed in Section \ref{sec:tracking}, we deduce that the trackable subspace in infinite time satisfies
\[
\mathcal{Y}_{\opNeu - \opNeu}(0,\infty) = \mathcal{G}_{2,1/\sqrt{2},0}(0,\infty) \cap C^\infty_{(0)}[0,\infty).
\]
For the trackable subspace in finite time, we need the reachable space of the control system \eqref{eq:heat_NN}, that is defined by the same equations as \eqref{eq:heat_NN} but with nonzero initial data disregarding the output $y$. From \cite[\S 3.2]{hartmann_orsoni}, this reachable space is given by
\[
\mathcal{R}_{\opNeu - \opNeu} = \{ f \in \mathcal{H}(\Omega): f \mbox{ odd} ,\quad f' \in L^2(\Omega) \}.
\]
Repeating the proof of Theorem \ref{theo:Y_finite_T}, one arrives to 
\[
\mathcal{Y}_{\opNeu - \opNeu}(0,T) = \left\lbrace y \in \mathcal{G}_{2,1/\sqrt{2},0}(0,T) \cap C^\infty_{(0)}[0,T]: \sum_{k=0}^\infty y^{(k)}(T) \frac{\zeta^{2k}}{(2k)!} \in A^2(\Omega) \right\rbrace.
\]
\subsubsection{Dirichlet-to-Neumann}
We consider now the system 
\begin{equation}\label{eq:heat_DN}
\left\lbrace \begin{array}{rclcc}
z_t(t,x) &=& z_{xx}(t,x),& 0 < x < 1,&t>0,\\
z(t,1) &=& u(t), \\
z(t,0) &=& 0,\\
z(0,x) &=& 0 ,  \\
y(t) &=& z_x(t,0).
\end{array}\right.
\end{equation}
When the initial data in \eqref{eq:heat_DN} is nonzero and disregarding the output $y$, it is a well-posed system on the state space $H^{-1}(0,1) := (H^1_0(0,1))'$, with reachable space
\[
\mathcal{R}_{\opDir - \opNeu} = \{ f \in \mathcal{H}(\Omega): f \mbox{ odd},\quad f \in L^2(\Omega) \}.
\]
Observe that when the initial data $z^0 \in H^{-1}(0,1)$ of \eqref{eq:heat_DN} is nonzero and the control $u$ is zero, the output $y(t) = z_x(t,0)$ may fail to be integrable around $t = 0$. But this does not affect our method as explained in Remark \ref{rem:no_adm_C}. Reasoning as in the previous subsubsection one sees that the transfer function of \eqref{eq:heat_DN} is given by 
\[
\mathbf{H}_{\opDir - \opNeu}(s) = \frac{\sqrt{s}}{\sinh \sqrt{s}},
\]
so that 
\[
\mathcal{Y}_{\opDir - \opNeu}(0,\infty) =  \mathcal{G}_{2,1/\sqrt{2},-1/2}(0,\infty) \cap C^\infty_{(0)}[0,\infty),
\]
and
\[
\mathcal{Y}_{\opDir - \opNeu}(0,T) = \left\lbrace y \in \mathcal{G}_{2,1/\sqrt{2},-1/2}(0,T) \cap C^\infty_{(0)}[0,T]:\sum_{k=0}^\infty y^{(k)}(T) \frac{\zeta^{2k+1}}{(2k+1)!} \in A^2(\Omega) \right\rbrace.
\]
\subsubsection{Dirichlet-to-Dirichlet}
Finally consider the system 
\begin{equation}\label{eq:heat_DD}
\left\lbrace \begin{array}{rclcc}
z_t(t,x) &=& z_{xx}(t,x),& 0 < x < 1,&t>0,\\
z(t,1) &=& u(t), \\
z_x(t,0) &=& 0,\\
z(0,x) &=& 0 ,  \\
y(t) &=& z(t,0).
\end{array}\right.
\end{equation}
When the initial data in \eqref{eq:heat_DD} is nonzero and disregarding the output $y$, it is a well-posed system on the state space $H^{-1}(0,1)$. We have
\[
\mathcal{R}_{\opDir - \opDir} = \{ f \in \mathcal{H}(\Omega): f \mbox{ even},\quad f \in L^2(\Omega) \}, \quad \mathbf{H}_{\opDir - \opDir}(s) = \frac{1}{\cosh \sqrt{s}},
\]
so that 
\[
\mathcal{Y}_{\opDir - \opDir}(0,\infty) = \mathcal{G}_{2,1/\sqrt{2},0}(0,\infty) \cap C^\infty_{(0)}[0,\infty),
\]
and
\[
\mathcal{Y}_{\opDir - \opDir}(0,T) = \left\lbrace y \in \mathcal{G}_{2,1/\sqrt{2},0}(0,T) \cap C^\infty_{(0)}[0,T]: \sum_{k=0}^\infty y^{(k)}(T) \frac{\zeta^{2k}}{(2k)!} \in A^2(\Omega) \right\rbrace.
\]
\subsection{Smooth control laws}
For a given integer $p \geq 1$, consider the case where the control $u$ belongs to $H^p_0(0,\infty)$, which is defined as 
\[
H^p_0(0,\infty) := \left\lbrace u \in H^p(0,\infty) : u(0) = u'(0) = \cdots = u^{(p-1)}(0) = 0 \right\rbrace.
\]
For such $u$, the relation \eqref{eq:laplace_rewrite} can be rewritten using
\[
\widehat{u^{(p)}}(s) = s^p \hat{u}(s).
\]
This allows to identify the set $\mathcal{Y}^p(0,\infty)$ of these signals that can be tracked with $H^p_0(0,\infty)$ controls. For finite time considerations, Theorem \ref{theo:Y_finite_T} can be adapted using the knowledge we have of the reachable space of the heat equation with smooth inputs, see \cite[\S 3.2]{hartmann_orsoni}. 

\subsection{Other point measurement}
For the system \eqref{eq:heat}, the output $y$ is the Dirichlet trace at $x = 0$. Consider the same problem but with measurement taken at some $0 < x_0 < 1$: 
\begin{equation}\label{eq:heat_x0}
\left\lbrace \begin{array}{rclcc}
z_t(t,x) &=& z_{xx}(t,x),& 0 < x < 1,&t>0,\\
z_x(t,1) &=& u(t), \\
z_x(t,0) &=& 0,\\
z(0,x) &=& 0 ,  \\
y(t) &=& z(t,x_0).
\end{array}\right.
\end{equation}
The proof of Theorem \ref{theo:charac_Y_Tinfty} can be easily adapted to cover this case. Let us denote by $\mathcal{Y}_{x_0}(0,\infty)$ the corresponding trackable subspace.
\begin{Theorem}
For every $0 < x_0 < 1$, the set $\mathcal{Y}_{x_0}(0,\infty)$ is constituted of these $\varphi \in C^\infty_{(0)}[0,\infty)$ such that 
\begin{equation}\label{eq:CNS_reg_interior}
\sum_{k=0}^{\infty} \left( \frac{\|{\varphi^{(k+1)}}\|_{L^2(0,\infty)}}{(1+k)^{3/4}(2k)!} \left(\frac{1-x_0}{\sqrt{2}}\right)^{2k}\right)^2 < \infty.     
\end{equation}
\end{Theorem}
The Gevrey radius $(1-x_0)/\sqrt{2}$ is the radius of the circle centered at $x_0$ and inscribed in $\Omega$. 
\begin{proof}
Let $u \in L^2(0,\infty)$ and denote $y$ the associated output of \eqref{eq:heat_x0}. Reasoning as in the proof of Theorem \ref{theo:charac_Y_Tinfty}\footnote{One could also reason as for \eqref{eq:heat_NN} considering the output $z_t(t,x_0)$} we deduce that $y$ is Laplace transformable, with absicssa of absolute convergence no greater than $0$, and satisfies
\[
\hat{y}(s) = \frac{\cosh(\sqrt{s}x_0)}{\sqrt{s}\sinh(\sqrt{s})}\hat{u}(s),\quad \opRe s > 0.
\]
We rewrites the above as 
\[
\psi(s)\hat{\dot{y}}(s) = \hat{u}(s),\quad \psi(s):= \frac{\sinh{\sqrt{s}}}{\sqrt{s}\cosh(\sqrt{s}x_0)}.
\]
The function $1/\psi$ is bounded on $\mathbb{C}_+$ from the Phragm\'en-Lindel\"of principle, hence $\dot{y} \in L^2(0,\infty)$, and we moreover have
\[
+ \infty > \int_\mathbb{R} |\hat{u}(i\xi)|^2 d\xi = \int_\mathbb{R} | \psi(i\xi)\hat{\dot{y}}(i\xi) | ^2 d\xi \asymp \int_\mathbb{R} \left| \mathcal{F}\dot{y}(\xi) \frac{e^{\sqrt{|\xi|}\frac{1-x_0}{\sqrt{2}}}}{1+\sqrt{|\xi|}}\right|^2 d\xi.
\]
This precisely means that $\dot{y} \in \mathcal{G}_{2,\frac{1-x_0}{\sqrt{2}},-1/2 }$, and from Theorem \ref{theo:plancherel_gevrey} we deduce \eqref{eq:CNS_reg_interior}.

Conversely, let $\varphi \in C^\infty_{(0)}[0,\infty)$ satisfy \eqref{eq:CNS_reg_interior}. Clearly it is enough to show that $\psi \hat{\dot{y}} \in \mathcal{H}^2(\mathbb{C}_+)$. We use Proposition \ref{prop:PWP}, the latter function has boundary values $\psi(i\xi) \mathcal{F}\dot{y}(\xi)$ which are square- integrable by hypothesis \eqref{eq:CNS_reg_interior} and Theorem \ref{theo:plancherel_gevrey}. It is therefore enough to check the condition on the support of the Fourier transform. To simplify the notations, we introduce the shorthand $\opspec T := \opsupp \mathcal{F}T$ for every $T \in \mathcal{S}'(\mathbb{R})$. We then have 
\begin{align*}
    \opspec \psi(i\xi) \mathcal{F}\dot{y}(\xi) &= \opspec \left( \frac{\sinh{\sqrt{i\xi}}}{\sqrt{i\xi}} \cdot \frac{\mathcal{F}\dot{y}(\xi)}{\cosh(\sqrt{i\xi}x_0)} \right) \\
    &\subset \opspec \left( \frac{\mathcal{F}\dot{y}(\xi)}{\cosh(\sqrt{i\xi}x_0)} \right) \\
    &\subset (\opspec \mathcal{F}\dot{y}(\xi)) + \left( \opspec \frac{1}{\cosh(\sqrt{i\xi}x_0)}\right) \\
    &\subset (-\infty,0] + (-\infty,0] \subset (-\infty,0],
\end{align*}
where the first inclusion is due to Lemma \ref{lem:spectrum} and 
\[
\opspec \frac{1}{\cosh(\sqrt{i\xi}x_0)} \subset (-\infty,0],
\]
is elementary by contour integration.
\end{proof}
In view of Theorem \ref{theo:Y_finite_T} one may hope for a characterization of $\mathcal{Y}_{x_0}(0,T)$ when $0 < T < \infty$. However, the flat-output method brings 
\[
z(t,x) = \sum_{k=0}^\infty y^{(k)}(t) \frac{(x-x_0)^{2k}}{(2k)!} + \sum_{k=0}^\infty \partial_t^k z_x(t,x_0) \frac{(x-x_0)^{2k+1}}{(2k+1)!},
\]
hence our method does not allow us to characterize $\mathcal{Y}_{x_0}(0,T)$. This suggests the investigation of the following system 
\begin{equation}\label{eq:heat_x02}
\left\lbrace \begin{array}{rclcc}
z_t(t,x) &=& z_{xx}(t,x),& -1 < x < 1,&t>0,\\
z(t,-1) &=& u_1(t), \\
z(t,1) &=& u_2(t),\\
z(0,x) &=& 0 ,  \\
y_1(t) &=& z(t,x_0),\\
y_2(t) &=& z_x(t,x_0).
\end{array}\right.
\end{equation}
For brevity we leave open the determination of the trackable subspace for the system \eqref{eq:heat_x02}.
\section{Interpolation in Gevrey classes of order $2$}
In this Section we prove Proposition \ref{prop:borel_gevrey} and Corollary \ref{coro:individual_loss}.
\begin{proof}[Proof of Proposition \ref{prop:borel_gevrey}]
We begin by proving the result in case $R = 1/\sqrt{2}$ and $\gamma = 0$. This case is special because the system ``Dirichlet-to-Dirichlet'' \eqref{eq:heat_DD} has its trackable subspace satisfying 
\[
\mathcal{Y}_{\opDir - \opDir}(0,\infty) = \mathcal{G}_{2,1/\sqrt{2},0}(0,\infty) \cap C^\infty_{(0)}[0,\infty).
\]
Moreover, its reachable space is given by 
\begin{equation}\label{eq:reachable_DirDir}
\mathcal{R}_{\opDir - \opDir} = \{ f \in \mathcal{H}(\Omega): f \mbox{ even},\quad f \in L^2(\Omega) \}.    
\end{equation}
Let us now show that 
\begin{equation}\label{eq:charac_borel_simple}
\mathcal{B}_{t_0} \mathcal{G}_{2,1/\sqrt{2},0}(I) = \left\lbrace (a_k): \sum_{k=0}^\infty a_k \frac{\zeta^{2k}}{(2k)!} \in A^2(\Omega) \right\rbrace.
\end{equation}
By cutoff arguments, there is no loss of generality in assuming that $I$ is bounded (see \cite[Lemma 3.7]{reachable_martin}). By translation we may further assume that $\inf I = 0$. Then, on the one hand, take $(a_k)$ in $\mathcal{B}_{t_0}\mathcal{G}_{2,1/\sqrt{2},0}(I)$ and fix $\varphi$ a pre-image. Owing to Corollary \ref{coro:loc_tracable} the function $\varphi$ is in $\mathcal{Y}_{\opDir-\opDir}(0,t_0)$ and thus satisfies 
\begin{equation}\label{eq:series_reachable}
\sum_{k=0}^\infty a_k \frac{\zeta^{2k}}{(2k)!} = \sum_{k=0}^\infty \varphi^{(k)}(t_0) \frac{\zeta^{2k}}{(2k)!} \in \mathcal{R}_{\opDir-\opDir}.
\end{equation}
This shows the direct inclusion of \eqref{eq:charac_borel_simple}. On the other hand, let $(a_k)$ be a sequence such that 
\begin{equation}\label{eq:une-serie-au-pif}
\sum_{k=0}^\infty a_k \frac{\zeta^{2k}}{(2k)!} \in A^2(\Omega).    
\end{equation}
By definition of the reachable space and \eqref{eq:reachable_DirDir}, there is $u \in L^2(0,t_0)$ such that the solution $z$ of \eqref{eq:heat_DD} with control $u$ reaches the state defined by the power series in \eqref{eq:une-serie-au-pif}, in time $t_0$. We extend the control $u$ by 0 on $(t_0,\infty)$, which extends $z$ for any positive times, what shows that $(a_k) \in \mathcal{B}_{t_0}\mathcal{G}_{2,1/\sqrt{2},0}(I)$. This completes the proof of \eqref{eq:charac_borel_simple}.

Acting on $(a_k)$ by scaling and shift, we deduce that 
\[
\mathcal{B}_{t_0} \mathcal{G}_{2,R,p}(I) = \left\lbrace (a_k): \sum^\infty a_{k+p} \frac{(\sqrt{2}R\zeta)^{2k}}{(2k)!} \in A^2(\Omega) \right\rbrace,\quad R > 0, \quad p \in \mathbb{Z}.
\]
To deal with the case $\gamma = p - 1/2$, $p \in \mathbb{Z}$, we repeat the same analysis for the system \eqref{eq:heat_DN}. We obtain 
\[
\mathcal{B}_{t_0} \mathcal{G}_{2,1/\sqrt{2},-1/2}(I) = \mathcal{B}_{t_0} \mathcal{Y}_{\opDir-\opNeu}(0,\infty) = \left\lbrace (a_k): \sum_{k=0}^\infty a_k \frac{\zeta^{2k+1}}{(2k+1)!} \in A^2(\Omega) \right\rbrace,
\] 
and 
\[
\mathcal{B}_{t_0} \mathcal{G}_{2,R,p-1/2}(I) = \left\lbrace (a_k): \sum^\infty a_{k+p} \frac{(\sqrt{2}R\zeta)^{2k+1}}{(2k+1)!} \in A^2(\Omega) \right\rbrace,
\]
which completes the proof of Proposition \ref{prop:borel_gevrey}.
\end{proof}
\begin{proof}[Proof of Corollary \ref{coro:individual_loss}]
We fix $(a_n)$ an arbitrary sequence of complex numbers and we introduce the quantity 
\[
\mathfrak{R}_a :=  \sup \{ R > 0: \exists \varphi \in \mathcal{G}_{2,R,0}(-1,1),\quad \forall n \in \mathbb{N},\quad \varphi^{(n)}(0) = a_n \}.
\]
From Proposition \ref{prop:borel_gevrey} we clearly have 
\[
\mathfrak{R}_a = \sup \left\lbrace R \geq 0: \sum_{k=0}^\infty a_k \frac{(\sqrt{2} R \zeta)^{2k}}{(2k)!} \in A^2(\Omega) \right\rbrace,
\]
hence we are left to show that $\mathfrak{R}_a = R_a$. On the one hand, for all $0 < R' < R$ we have $G^{2,R}[-1,1] \subset \mathcal{G}_{2,R',0}(-1,1)$, hence $R_a \leq \mathfrak{R}_a$. On the other hand, let $0 < R < \mathfrak{R}_a$ and $\varphi \in \mathcal{G}_{2,R,0}(-1,1)$ interpolating $(a_n)$. From \cite[Lemma 3.7]{reachable_martin} we can further assume that $\varphi$ is flat at $\pm 1$, so that in particular the Poincaré inequality applies to $\varphi$ and and all its successive derivatives. We easily deduce that for all $0 < R' < R$ there holds $\varphi \in G^{2,R'}[-1,1]$, hence $R' \leq R_a$. We deduce that $\mathfrak{R}_a \leq R_a$, which concludes the proof. 
\end{proof}
In Proposition \ref{prop:borel_LR} we have given an example of a sequence $(a_k)$ whose Gevrey radius is $1$, but which cannot be interpolated by a function that is Gevrey of radius $1/\sqrt{2}$. Corollary \ref{coro:individual_loss} allows to obtain the other extreme case: the sequence 
\[
a_k := \frac{(2k)! 2^k i^k}{\sqrt{k}},\quad a_0 := 0,
\]
can be interpolated by a function which is Gevrey of order $2$ and radius $1/\sqrt{2}-\epsilon$, for all $\epsilon > 0$ (see \eqref{eq:asymptotics_polylog}).

\appendix
\section{Appendix: more on the interpolation problem}
\label{app}
In this Appendix, we provide a detailed analysis of the interpolation problem in Gevrey classes of order $s > 1$, notably by completing the proof of Theorems 1 and 1a of \cite{mitjagin}, which is only sketched in the cited paper. \newline
\newline
Let us first clarify Remark \ref{rem:MMR}: in \cite[Theorem 3.2]{reachable_martin}, it is claimed that for every sequence $(a_n)$ such that 
\[
    |a_n| \lesssim \frac{(2n)!}{R_0^{2n}},
\]
and for every $\epsilon > 0$, the sequence $(a_n)$ can be interpolated by a function $\varphi \in C^\infty(\mathbb{R})$ such that
\[
    |\varphi^{(n)}(t)| \lesssim \frac{(2n)!}{(\rho^{MRR}R_0 - \epsilon)^{2n}},
\]
where  
\[
\rho^{MRR}:= \exp\left( -\frac{1}{2e} \right)\approx 0.832 > 0.707 \approx \frac{1}{\sqrt{2}} = \rho_2.
\]
This contradicts Proposition \ref{prop:borel_LR} and \cite[Theorem 1a]{mitjagin}, which impose the loss factor $\rho^{MRR}$ to be $\geq 1/\sqrt{2}$. 

We believe that, unfortunately, the proof of \cite[Theorem 3.2]{reachable_martin} is flawed. Indeed, to go further into the details of the proof in \cite{reachable_martin}, the essence of their Theorem 3.2 lies in their Proposition 3.6. To prove this last Proposition, the authors build $f(x) = \sum_p d_p \varphi_p(x) x^p/p!$ where each $\varphi_p$ is built according to their Corollary 3.5. The constant $C$ appearing in their Corollary 3.5 implicitly depends on $\delta$ and on the sequence $(a_k)$. However, in the proof of their  Proposition 3.6, this dependency is not taken into account, and the constant $C$ appearing in the estimate of $\varphi_p^{(j)}(x)$ page 195 \textit{a priori} depends at least on $p$, which prevents one from concluding as the authors claim. 

Note further that \cite[Proposition 3.6]{reachable_martin} may be used to show that the interpolation problem can be solved in Gevrey classes of order $s>1$, with loss factor 
\[
\rho_s^{MRR} := \exp\left( - \frac{1}{es} \right). 
\]
As explained above, we believe that the proof of \cite[Proposition 3.6]{reachable_martin} is flawed. In fact, for $1 < s < 3$ we have $\rho_s^{MRR}> \rho_s$, hence a contradiction with Mitjagin's results. For $s>4$ we have $\rho_s^{MRR} < \rho_s$, hence $\rho_s^{MRR}$ is a valid loss factor, although sub-optimal. 

We believe this is a minor inaccuracy. Indeed, the result itself of \cite[Theorem 3.2]{reachable_martin} might be corrected to yield a valid result, by considering the loss factor of \cite[Theorem 1]{mitjagin}, whereas the proof of \cite[Theorem 3.2]{reachable_martin} might be corrected to yield a valid result, probably with a suboptimal loss factor, see \cite[Theorem 3.6]{petzsche}. As far as we are concerned, whenever Theorem 3.2 or Proposition 3.6 from \cite{reachable_martin} is used, the resulting assertion can be corrected by merely modifying a numerical constant. Moreover, this does not detract from the importance of the article \cite{reachable_martin}, which is the first to successfully describe the reachable states of the heat equation using holomorphic functions.
\subsection{Generalities}
In this subsection we study the interpolation problem in a rather general setting. In order to remain close to Mitjagin's notations we change our convention for Gevrey functions. This is a minor modification and we shall explain how to relate the loss factors between the two conventions. 
\begin{Definition}\label{def:carl}
Let
\[
-\infty < a < b < \infty,\quad c \in [a,b], \quad A > 0,\quad 1 \leq \Gamma < \infty, \quad 1 < s < \infty.
\]
We say that the interpolation problem associated with $(a,b,c,A,s)$ is solvable with loss $\Gamma$ if, for every sequence of complex numbers $(a_n)$ such that 
\[
\exists C > 0,\quad \forall n \in \mathbb{N},\quad |a_n| \leq C A^n n^{ns},
\]
and for every $\epsilon > 0$, there exists a function $\varphi \in C^\infty[a,b]$ such that 
\[
\forall n \in \mathbb{N},\quad \varphi^{(n)}(c) = a_n,
\]
and 
\[
\exists C' > 0, \quad \forall n \in \mathbb{N},\quad \forall t \in [a,b],\quad |\varphi^{(n)}(t)| \leq C' (\Gamma A+\epsilon)^n n^{ns}.
\]
\end{Definition}
If $\Gamma_s$ stands for the optimal loss factor for the above definition and $\rho_s$ stands for the optimal loss factor with the conventions used in \eqref{eq:sequence_gevrey} and \eqref{eq:interpolate_gevrey}, an application of Stirling's formula yields $\rho_s = \Gamma_s^{-1/s}$. 

Definition \ref{def:carl} \textit{a priori} depends on the parameters $(a,b,c,A,s)$. We claim that actually it does not depend on $a,b,c,A$ when $s$ is fixed, in the sense that changing $a,b,c$ and $A$ does not affect the loss factor $\Gamma$.
\begin{Proposition}\label{prop:carl_indep}
Fix $s > 1$ and assume that the interpolation problem associated with $(a,b,c,A,s)$ is solvable with loss $\Gamma$. Then, the interpolation problem associated with $(a',b',c',A',s)$ is solvable with the same loss factor $\Gamma$. 
\end{Proposition}
\begin{proof}
Assume that the interpolation problem associated with $(a,b,c,A,s)$ is solvable with loss factor $\Gamma$. We show that for every $(a',b',c',A',s)$ as above, the associated problem is solvable with loss $\Gamma$. Clearly, for any translation $\tau \in \mathbb{R}$, the problem associated with $(a+\tau,b+\tau,c+\tau,A,s)$ is solvable with the same loss factor. We may therefore  assume that $a = a'=0$. Let us then use a dilatation: if $\varphi \in C^\infty[0,b]$ is such that 
\[
|\varphi^{(n)}(t)| \leq C' (\Gamma A+\epsilon)^n n^{ns},
\] 
then for $\lambda > 0$ we consider $\varphi_\lambda(t):= \varphi(\lambda t)$, which is such that 
\[
\varphi_\lambda \in C^\infty[0,b/\lambda],\quad \varphi_\lambda^{(n)}(t) = \lambda^n \varphi^{(n)}(t),\quad |\varphi_\lambda^{(n)}(t)| \leq C'(\Gamma A \lambda + \epsilon \lambda)^n n^{ns}.
\]
Thus the problem of parameters $(0,b/\lambda,c/\lambda,A \lambda,s)$ is solved with loss $\Gamma$. We can therefore assume that $b=b'=1$. \par 
We then claim that if the problem of parameters $(0,1,c,A,s)$ is solvable with loss $\Gamma$ for $0 < c < 1$, then for every $c' \in (0,1)$ the problem of parameters $(0,1,c',A,s)$  is solvable with loss $\Gamma$. For, let $\epsilon>0$ and $(a_n)$ be any sequence such that 
\[
|a_n| \leq C  A^n n^{ns}.
\]
We invoke the solvability of the problem associated with $(0,1,c,A,s)$ to take $\varphi \in C^\infty[0,1]$ such that 
\[
\varphi^{(n)}(c) = a_n,\quad |\varphi^{(n)}(t)| \leq C'  (\Gamma A+\epsilon)^n n^{ns}.
\]
We then consider 
\[
\psi(t):= \varphi(t-c'+c)\chi(t),
\]
where $\chi \in C^\infty(\mathbb{R})$ is Gevrey of order $r\in (1,s)$, is constant to 1 on a small neighborhood of $c'$, and flat at $c'-c$ and $1+c'-c$. Since $1<r<s$, the constructed function $\psi$ has the same Gevrey regularity as $\varphi$ in the sense that 
$$|\psi^{(n)}(t)| \leq C'' (\Gamma A+\epsilon)^n n^{ns},$$
(see \cite[Lemma 3.7]{reachable_martin}). Moreover, since $\chi$ is $0$ outside of $(c'-c, 1+c'-c)$, $\psi$ is well-defined on $\mathbb R$ (so it is also defined on $[0,1]$), and has the same derivatives at $c'$ than $\varphi$ has at $c$. Thus the problem of parameters $(0,1,c',A,s)$ is solvable, with the same loss $\Gamma$.\par 
Next we show that the cases $c=0$ or $c=1$ are equivalent. Assume for instance that the problem of parameters $(0,1,0,A,s)$ is solvable with loss $\Gamma$. For any sequence $(a_n)$ such that 
\[
|a_n| \leq C A^n n^{ns},
\]
we invoke its solution for the sequence $((-1)^na_n)$ to obtain $\varphi \in C^\infty[0,1]$ such that 
\[
\varphi^{(n)}(c) = (-1)^n a_n,\quad |\varphi^{(n)}(t)| \leq C' (\Gamma A+\epsilon)^n n^{ns}.
\]
We then put $\psi(t):= \varphi(1-t)$, which has the same Gevrey regularity as $\varphi$ and is such that
\[
\psi^{(n)}(1) = (-1)^n \varphi^{(n)}(0) = a_n.
\]
The converse implication is proved in the same way, so that $c=0$ and $c=1$ are equivalent. \par 
Assume now that the problem of parameters $(0,1,0,A,s)$ is solvable with loss $\Gamma$. We show that, for every $0 < c < 1$ the problem of parameters $(0,1,c,A,s)$ is solvable with loss $\Gamma$. In fact we will show that $(-1,1,0,A,s)$ is solvable with loss $\Gamma$, which is easily seen as a stronger statement, by translation and dilatation arguments. We let $(a_n)$ be a sequence with 
\[
|a_n| \leq C A^n n^{ns}.
\]
We introduce the ``odd'' and ``even'' sequences
\[
a_n^o = \left\lbrace \begin{array}{rl}
a_{2k+1}, & n=2k+1, \\
0, & n = 2k,
\end{array}\right.,\quad a_n^e = \left\lbrace \begin{array}{rl}
a_{2k}, & n=2k, \\
0, & n = 2k+1,
\end{array}\right.
\]
and take $\varphi,\psi \in C^\infty[0,1]$ so that 
\[
|\varphi^{(n)}(t)|,|\psi^{(n)}(t)| \leq C' (\Gamma A+\epsilon)^n n^{ns},\quad \varphi^{(n)}(0) = a_n^o,\quad \psi^{(n)}(0) = a_n^e.
\]
We then consider $\tilde{\varphi}$ the odd extension of $\varphi$ and $\tilde{\psi}$ the even extension of $\psi$. Clearly, they are both $C^\infty$ on $[-1,0]$ and $[0,1]$, and on each of these two intervals they satisfy the  Gevrey estimate 
\[
|\tilde{\psi}^{(n)}(t)|,|\tilde{\varphi}^{(n)}(t)| \leq C' (\Gamma A+\epsilon)^n n^{ns}.
\]
To check that $\tilde{\varphi}$ and $\tilde{\psi}$ satisfy the required Gevrey estimate on $[-1,1]$, it is thus enough to check that their derivatives at $0^\pm$ both match. We have 
\[
\tilde{\varphi}^{(n)}(0^+) = \varphi^{(n)}(0) = a_n^o,\quad \tilde{\varphi}^{(n)}(0^-) = (-1)^{n+1}\varphi^{(n)}(0) = (-1)^{n+1}a_n^o,
\]
and these two values agree. Similarly,
\[
\tilde{\psi}^{(n)}(0^+) = \psi^{(n)}(0) = a_n^e,\quad \tilde{\psi}^{(n)}(0^-) = (-1)^{n}\psi^{(n)}(0) = (-1)^{n}a_n^e,
\]
so that $\phi:= \tilde{\varphi} + \tilde{\psi}$ solves the problem. \par 
We can thus put $c=c'=0$, hence we are left to show that if the problem $(0,1,0,A,s)$ is solvable with loss $\Gamma$, then, for every $A' > 0$, the problem $(0,1,0,A',s)$ is solvable with loss $\Gamma$. Assume that the problem associated with $(0,1,0,A,s)$ is solvable with loss $\Gamma$, using the dilatation $\varphi_\lambda$ we see that for every $\lambda > 0$, the problem associated with $(0,1/\lambda,0,A\lambda,s)$ is solvable with loss $\Gamma$. We distinguish two cases: if $\lambda \leq 1$, then the problem associated with $(0,1, 0, \Gamma, A\lambda,s)$ is trivially solvable. If on the contrary $\lambda > 1$, using a cutoff that is Gevrey of order $1 < r < s$, constant to $1$ near $0$ and flat at $1/\lambda$, we see that the problem associated with $(0,1, 0, \Gamma, A\lambda,s)$ is also solvable with loss $\Gamma$, as in the case of the equivalence for the parameter $c$. Since $\lambda > 0$ is arbitrary, this finishes the proof. 
\end{proof}
Therefore, the only relevant numbers are $\Gamma$ and $s$. 
\subsection{A proof of Mitjagin's Theorem}\label{subsec:mit}
We now provide a detailed proof of Theorems 1 and 1a of \cite{mitjagin}. Recall the notation
\[
\Gamma_\beta := \cos^{-\beta} \frac{\pi}{2\beta},\quad \beta > 1.
\]
We begin by showing that $\Gamma_\beta$ is an admissible loss factor. 
\begin{Theorem}\label{theo:mitjagin1}
Let $\beta > 1$ and $B,\epsilon > 0$ be fixed. Then for every sequence $(a_n)$ such that 
\[
\exists C > 0,\quad \forall n \in \mathbb{N},\quad |a_n| \leq C B^n n^{n \beta},
\]
there exists $\varphi \in C^\infty[-1,1]$ such that 
\[
\exists C' > 0,\quad \forall t \in [-1,1],\quad \forall n \in \mathbb{N}, \quad |\varphi^{(n)}(t)| \leq C' (\Gamma_\beta B +\epsilon)^n n^{n \beta}
\]
together with
\[
\forall n \in \mathbb{N},\quad \varphi^{(n)}(0) = a_n.
\]
\end{Theorem}
In order to show this theorem we need some preparation. For given parameters $B > 0$ and $\beta > 1$ we define $C^\infty_{\beta B}$ as the set of these $\phi \in C^\infty[-1,1]$ such that 
\[
\|\phi\|_{C^\infty_{\beta B}}:= \sup_{n \in \mathbb N}\sup_{t \in [-1, 1]} \frac{|\phi^{(n)}(t)|}{B^n n^{n\beta}} < \infty.
\]
This defines a Banach space. 

A functional $f$ on $C^\infty_{\beta B}$ is said to be concentrated (at 0) if it vanishes on any $\phi \in C^\infty_{\beta B}$ that is equal to $0$ around $t=0$. For such a functional $f$ we construct an entire series, roughly speaking the Fourier transform of $f$, in the following Lemma. We introduce a function $h \in C^\infty(\mathbb{R})$ such that 
\[
|h^{(n)}(t)| \leq CA^n n^{n \gamma},\quad h(t) = 1 ~~ (|t| < 1/2),\quad h(t) = 0 ~~ (|t| > 1),
\]
for some constants $1 < \gamma < \beta$ and $A,C > 0$. 
\begin{Lemma}\label{lem:mitjagin} Let $f$ be a concentrated functional on $C^\infty_{\beta B}$, $\epsilon > 0$ and $h$ as above. We then have 
\begin{enumerate}
\item For every $s \in \mathbb{C}$, the function 
\[
\phi_s(t):= e^{ist} h(t/\epsilon),
\]
is in $C^\infty_{\beta B}$, and verifies the estimate: for any $\delta > 0$, we have 
\begin{equation}\label{eq:growth_phi}
\|\phi_s\|_{C^\infty_{\beta,B}} \leq C_1 \exp \left[\epsilon |\tau| + \frac{\beta}{e} \left( \frac{1+\delta}{B} r \right)^{1/\beta} \right],
\end{equation}
where $r = |s|$, $s = \xi + i \tau$, and $C_1 = C_1(C,A,\epsilon,\delta,\beta,B)$. 
\item The function $\Phi(s):= \langle f , \phi_s \rangle$ is well-defined, entire on $\mathbb{C}$, and does not depend on $\epsilon > 0$ or $h$ as above. 
\item The function $\Phi$ satisfies the estimate 
\begin{equation}\label{eq:growth_ell}
|\Phi(s)|\leq  \|f\| C_1 \exp\left(\frac{\beta}{e} \left( \frac{1+\delta}{B} r \right)^{1/\beta} \frac{1}{\cos(\pi /(2 \beta))}\right).
\end{equation}
\item We have the estimation of the Taylor coefficients of $\Phi$
\begin{equation}\label{eq:growth_coef}
|f_n| \leq C_1 \|f\| \left( \frac{1+\delta}{B \cos^{\beta} \frac{\pi}{2 \beta}}\right)^n n^{-n\beta}, \quad \Phi(s) = \sum_{n=0}^\infty f_n s^n.
\end{equation}
\end{enumerate} 
\end{Lemma}
\begin{proof}
\begin{enumerate}
\item The function $\phi_s$ is clearly in  $C^\infty[-1,1]$, the goal is thus to estimate its derivatives. For this we use the Leibniz formula 
\[
\phi_s^{(n)}(t) = \sum_{k=0}^n \left( \begin{array}{c}
n \\
k
\end{array}\right) (is)^ke^{ist} \frac{1}{\epsilon^{n-k}} h^{(n-k)} \left( \frac{t}{\epsilon} \right).
\]
Since $h \equiv 0$ outside $[-1,1]$, we deduce that
\[
\left| e^{ist} h^{(p)} \left( \frac{t}{\epsilon} \right) \right| \leq e^{\epsilon |\tau| } C A^p p^{\gamma p},\, p\in\mathbb N.
\] 
We deduce that
\begin{align*}
|\phi_s^{(n)}(t)| &\leq \sum_{k=0}^n \left( \begin{array}{c}
n \\
k
\end{array}\right) r^k e^{\epsilon |\tau|} C \left( \frac{A}{\epsilon} \right)^{n-k} (n-k)^{(n-k) \gamma} \\
&\leq Ce^{\epsilon |\tau| } \sum_{k=0}^n \left( \begin{array}{c}
n \\
k
\end{array}\right) r^k \left( \frac{A}{\epsilon} \right)^{n-k} n^{(n-k) \gamma} \\
&= Ce^{\epsilon |\tau| } \left( r + \frac{An^\gamma}{\epsilon}\right)^n,
\end{align*}
where recall that $r = |s|$. In view of the definition of the $ \|\cdot\|_{C^\infty_{\beta B}} $-norm, the estimate \eqref{eq:growth_phi} thus boils down to proving that
\[
\forall \delta > 0,\quad \exists C_1 > 0,\quad \forall n \in \mathbb{N},\quad \forall r > 0,\quad  \left( r + \frac{An^\gamma}{\epsilon}\right)^n
 \leqslant C_1  B^n n^{n\beta} \exp \left[ \frac{\beta}{e} \left( \frac{1+\delta}{B} r \right)^{1/\beta} \right].
\]
If $n = 0$ the estimate is obvious (put $C_1 = 1$), hence we may take $n \geq 1$ without loss of generality. We will obtain the above estimate by considering separately three ranges for the parameters $r$ and $n$. We introduce another parameter $\gamma < \lambda < \beta$, and the first range we consider is $r \leq n^\lambda$. For such $r,n$ have
\[
\left( r + \frac{An^\gamma}{\epsilon}\right)^n \leq \left( n^\lambda + \frac{An^\gamma}{\epsilon}\right)^n = n^{\lambda n} \left( 1 + \frac{An^{\gamma-\lambda}}{\epsilon}\right)^n \leq n^{\lambda n} C(\epsilon,A,\gamma,\lambda),
\]
where $C(\epsilon,A,\gamma,\lambda)$ is a constant only depending on $\epsilon,A,\gamma,\lambda$. Because $\lambda < \beta$, the term $n^{\lambda n}$ is negligible  comapred to $B^nn^{n\beta}$ as $n \rightarrow \infty$, hence 
\[
\left( r + \frac{An^\gamma}{\epsilon}\right)^n \leq C(\epsilon,A,B,\gamma,\lambda,\beta) B^n n^{n\beta},\quad r \leq n^\lambda.
\]
We then assume that $r \geq n^\lambda$, and compute 
\begin{align*}
\left( r + \frac{An^\gamma}{\epsilon} \right)^n &= \left( r + \frac{A(n^\lambda)^{\gamma/\lambda}}{\epsilon} \right)^n \\
&\leq \left( r + \frac{Ar^{\gamma/\lambda}}{\epsilon} \right)^n \\
&= r^n \left(  1 + \frac{Ar^{\frac{\gamma}{\lambda}-1}}{\epsilon} \right)^n.
\end{align*}
Because $0 < \gamma/\lambda < 1$, we have 
\[
\forall \delta > 0,\quad \exists R = R(A,\epsilon,\delta,\gamma/\lambda) > 0,\quad \forall r > R,\quad 1 + \frac{Ar^{\frac{\gamma}{\lambda}-1}}{\epsilon} \leq 1+ \delta.
\]
Fix $\delta >0$ arbitrary and introduce an $R > 0$ as above. The second regime we consider is $r \geq \max(n^\lambda,R)$, for which we find
\[
\left( r + \frac{An^\gamma}{\epsilon} \right)^n \leq r^n (1+\delta)^n.
\]
We then fix $r \geq \max(n^\lambda,R)$ and compute
\[
\sup_{x > 0} \frac{r^x(1+\delta)^x}{B^x x^{x \beta}} = \sup_{y > 0} \frac{D^y}{y^y},\quad D := \left( \frac{r(1+\delta)}{B}\right)^{1/\beta} \beta, \quad y = x \beta. 
\]
The above supremum is reached at $y = D/e$, for which
\[
\sup_{x > 0} \frac{r^x(1+\delta)^x}{B^x x^{x \beta}} = e^{D/e} = \exp\left[ \frac{1}{e} \left( \frac{r(1+\delta)}{B}\right)^{1/\beta} \beta \right].
\]
Thus for $n \geq 1$ and $r \geq \max(n^\lambda,R)$ we find 
\[
\left( r + \frac{An^\gamma}{\epsilon} \right)^n \leq \exp\left[ \frac{1}{e} \left( \frac{r(1+\delta)}{B}\right)^{1/\beta} \beta \right] B^n n^{n \beta}.
\]
The third and last regime we investigate is when $n \geq 1$ and $n^\lambda \leq r \leq \max(R,n^\lambda)$. If $R < n^\lambda$ we are back to the first regime, hence we are left with $n^\lambda \leq R$. There are only finitely many such values of 
$n$, for which 
\[
\left( r + \frac{An^\gamma}{\epsilon} \right)^n \leq C(A,\epsilon,\lambda,\gamma,R) = C(A,\epsilon,\lambda,\gamma,\delta).
\]
\item The fact that $\Phi$ is well-defined as a function $\mathbb{C} \rightarrow \mathbb{C}$ follows from the previous point. It is elementary to check that it is holomorphic on $\mathbb{C}$. Because $f$ is concentrated, the function $\Phi$ does not depend on $h$ or $\epsilon$. 
\item In order to show \eqref{eq:growth_ell}, we start from \eqref{eq:growth_phi} and deduce
\[
|\Phi(s)| = |\langle f , \phi_s \rangle | \leq \| f \| \|\phi_s\|_{C^\infty_{\beta,B}} \leq\|f\|  C_1 \exp \left[\epsilon |\tau| + \frac{\beta}{e} \left( \frac{1+\delta}{B} r \right)^{1/\beta} \right]
\] 
where, as recalled, we constantly use the notations $r = |s|$, $s = \xi + i \tau$.
In particular, the function $\Phi$ satisfies 
\begin{equation}\label{flp0}
\forall \epsilon > 0,\quad \exists C_\epsilon > 0,\quad \forall s \in \mathbb{C},\quad |\Phi(s)| \leq C_\epsilon e^{\epsilon r},
\end{equation}
\textit{i.e.} it is of order $1$ and of type $0$ on $\mathbb{C}$. Note however that on the real axis we have a better estimate, of the form
\begin{equation}\label{flp00}
\exists C,K > 0,\quad \forall \xi \in \mathbb{R},\quad |\Phi(\xi)| \leq C e^{K |\xi|^{1/\beta}},
\end{equation}
\textit{i.e.} $\Phi$ is of finite type for the order $\leq 1/\beta$ on $\mathbb{R}$. We then use the following version of the Phragm\'en-Lindel\"of's method. 
\begin{leftbar} \noindent
Let $0< \rho < 1$ and $M,K \geq 0$. Assume that $g \in \mathcal{H}(\mathbb{C}_+) \cap C^0(\mathbb{C}_+ \cup i \mathbb{R})$ is of order 1 with minimal type 0, together with
\[\forall y \in \mathbb{R},\quad |g(iy)| \leq Me^{K |y|^\rho}.
\]
Therefore, for every $z \in \mathbb{C}_+$ we have
\[
|g(z)| \leq M \exp\left( \frac{K |z|^\rho}{\cos(\pi \rho/2)}\right) .
\]
\end{leftbar}
Let us show this version of the Phragm\'en-Lindel\"of principle. We will use the following ``standard'' version of the Phragmén-Lindelöf principle, see \cite[\S III.C]{K1}.
\begin{leftbar} \noindent
Let $\alpha \geq 1/2$, $\Omega = \{ re^{i\theta} : r > 0,\quad |\theta| < \frac{\pi}{2 \alpha} \}$, and $g \in \mathcal{H}(\Omega) \cap C^0(\overline{\Omega})$ of order $\rho < \alpha$, such that  
\[|g(re^{\pm i \frac{\pi}{2 \alpha}})| \leq M.
\]
Then, for every $z \in \Omega$ there holds
\[
|g(z)| \leq M .
\]
\end{leftbar}
To show the less standard version of the Phragmén-Lindelöf's meghod we fix a function $g$ as in the thesis. For every $\kappa > 0$ we consider the function 
\[
g_\kappa(z) = \exp \left( -\frac{K}{\cos(\pi \rho/2)}z^\rho - \kappa z\right)g(z) \in \mathcal{H}(\mathbb{C}_+) \cap C^0(\mathbb{C}_+ \cup i \mathbb{R}).
\]
We apply the ``standard'' Phragmén-Lindelöf's method quoted above to the function $g_\kappa$ on the domain $\Omega := \{ x+iy : x,y > 0 \}$. Note that the quoted result is invariant by rotation, hence we may see $\Omega$ as a sector with angle $\pi/2$, and the parameter $\alpha$ is set to $2$. On $\Omega$ the function $g_\kappa$ is of order $\rho = 1$, which is strictly less than $\alpha$. Moreover, on the boundary of $\Omega$ lying on the imaginary axis there holds 
\[
|g_\kappa(iy)| =  \exp \left( -\frac{K}{\cos(\pi \rho/2)} \opRe ((iy)^\rho)\right)|g(iy)| \leq e^{-Ky^\rho}Me^{Ky^\rho} = M.
\]
On the boundary of $\Omega$ lying on the real axis we proceed as follows: the function $g$ is of order $1$ and of type $0$, hence 
\[
\forall \eta > 0,\quad \exists C_\eta >,\quad \forall z \in \mathbb{C}_+,\quad |g(z)| \leq C_\eta e^{\eta |z|}. 
\]
We take $\eta = \kappa$, so that for every $x > 0$, 
\[
|g_\kappa(x)| \leq \exp \left( -\frac{K}{\cos(\pi \rho/2)} x^\rho - \kappa x \right)C_\kappa e^{\kappa x} \leq C_\kappa.
\]
The function $|g_\kappa|$ is thus bounded above on $\Omega$ by the constant $\max(M,C_\kappa)$, and the quoted Phragmén-Lindelöf's method yields that 
\[
\forall z \in \Omega,\quad |g_\kappa(z)| \leq \max(M,C_\kappa).
\]
The same reasoning can be made on the domain $\{ x+iy : y < 0 < x \}$, so that 
\[
\forall z \in \mathbb{C}_+,\quad |g_\kappa(z)| \leq \max(M,C_\kappa).
\]
The function $|g_\kappa|$ is thus bounded $\mathbb{C}_+$, hence from the quoted Phragmén-Lindelöf's method (or rather, the extended maximum principle, see \cite[III.C]{K1}),
\[
\sup_{z \in \mathbb{C}_+} |g_\kappa(z)| \leq \sup_{z \in \partial \mathbb{C}_+} |g_\kappa(z)| \leq M.
\]
Now fix $z \in \mathbb{C}_+$, for every $\kappa > 0$ we have $|g_\kappa(z)| \leq M$ and 
\[
g_\kappa(z) \rightarrow \exp \left( -\frac{K}{\cos(\pi \rho/2)}z^\rho\right)g(z),\quad \kappa \rightarrow 0^+.
\]
Thus
\[
\exp\left( -\frac{K}{\cos(\pi \rho/2)}\opRe(z^\rho)\right)|g(z)| \leq M,
\]
which shows the claimed version of the Phragmén-Lindelöf principle. This shows \eqref{eq:growth_ell}. 
\item We apply the Cauchy inequality on a disc of radius $0 < r < \infty$ to obtain 
\[
|f_n| = \left| \frac{\Phi^{(n)}(0)}{n!} \right| \leq \frac{1}{r^n} \max_{|s| = r} |\Phi(s)| \leq \frac{1}{r^n}\|f\| C_1 \exp \left[\frac{\beta}{e\cos(\pi/2\beta)} \left( \frac{1+\delta}{B} r \right)^{1/\beta} \right],
\]
where the second inequality is due to \eqref{eq:growth_ell}. Put 
\[
\mu := \frac{\beta}{e\cos(\pi/2\beta)} \left( \frac{1+\delta}{B} \right)^{1/\beta}>0,
\]
so that the previous estimate reads
\[
|f_n| \leq \|f\| C_1 \frac{e^{\mu r^{1/\beta}}}{r^n}.
\]
When $n \geq 1$, the above right-hand side reaches its minimal value with respect to $r$ when 
\[
r = \left( \frac{n \beta}{\mu}\right)^\beta,
\]
at which 
\[
|f_n| \leq \|f\| C_1 \left( \frac{e\mu}{n\beta} \right)^{n \beta} = C_1 \|f\| \left( \frac{1+\delta}{B \cos^{\beta} \frac{\pi}{2 \beta}}\right)^n n^{-n\beta}.
\]
\end{enumerate}
\end{proof}
We note that $\Phi$ does not depend on $h$ or $\epsilon$, hence taking $\epsilon = 2$ brings 
\[
\Phi(s) = \langle f(t) , e^{ist} \rangle.
\]
We now make the final preparations to show Theorem \ref{theo:mitjagin1}. For parameters $\beta > 1$ and $B > 0$ we introduce the Hilbert space $L^2_{\beta B}$ constituted of these $\phi\in C^\infty[-1,1]$ such that 
\[
\|\phi\|_{L^2_{\beta B}}^2:= \sum_{n=0}^\infty \int_{-1}^{+1} \left|\frac{\phi^{(n)}(t)}{{ B}^n n^{n\beta}} \right|^2 dt<+\infty.
\]
It is elementary to check that for every $\epsilon > 0$ we have 
\begin{equation}\label{eq:embed_func}
C^\infty_{\beta B} \subset L^2_{\beta B} \subset C^\infty_{\beta,B+\epsilon},
\end{equation}
where the inclusions are continuous and dense\footnote{That the first inclusion is bounded is trivial. That the second inclusion is bounded can be shown using the Sobolev embeding $H^1(-1,1) \subset C[-1,1]$. That the inclusions are dense is shown by mollifying with a kernel function lying in $C^\infty_{\gamma C}$ for some $1 < \gamma < \beta$ and $C >0$.}. We will need a similar construction for our spaces of sequences. For a sequence $(w_n)$ of positive numbers we define $K(w_n)$ as the set of these sequences $a = (a_n)$ such that 
\[
\|a\|_{K(w_n)} := \sup_{n \in \mathbb{N}} \frac{|a_n|}{w_n} < \infty.
\]
This defines a Banach space. For parameters $\beta > 1$ and $B >0$ we define $\ell^2_{\beta B}$ as the set of these sequences $a$ such that 
\[
\| a\|_{\ell^2_{\beta B}} ^2:= \sum_{n=0}^\infty \left|\frac{a_n}{B^n n^{n\beta}} \right|^2 < \infty.
\]
This defines a Hilbert space, and for every $\epsilon > 0$ we have the continuous and dense inclusions 
\begin{equation}\label{eq:embed_seq}
\ell^2_{\beta B} \subset K(n^{n\beta}B^n) \subset \ell^2_{\beta,B+\epsilon}.
\end{equation}
Below is the proof of Theorem \ref{theo:mitjagin1}, the proof we give is not exactly the one given by Mitjagin. Actually, it is longer, but it will make the proof of the optimality of $\Gamma_\beta$ more natural. 
\begin{proof}[Proof of Theorem \ref{theo:mitjagin1}]
Fix $\epsilon,B > 0$ and $\beta > 1$. We consider the Borel operator 
\[
\mathcal{B}_{\infty}^\epsilon : C^\infty_{\beta, \Gamma_\beta B+ \epsilon} \rightarrow K((\Gamma_\beta B + \epsilon)^n n^{n\beta}),\quad \mathcal{B}_{\infty}^\epsilon\phi = (\phi^{(n)}(0))_{n=0}^\infty,
\]
and the inclusion 
\[
i_{\infty}^\epsilon : K(B^n n^{n\beta}) \rightarrow K((\Gamma_\beta B + \epsilon)^n n^{n\beta}).
\]
The claim is precisely that 
\[
\forall \epsilon > 0,\quad \opRange i_\infty^\epsilon \subset \opRange \mathcal{B}_\infty^\epsilon.
\]
We will show this by duality arguments, hence it will be convenient to replace the domains and codomains of $\mathcal{B}_\infty^\epsilon$ and $i_\infty^\epsilon$ by Hilbert spaces\footnote{The Douglas Lemma does not hold on arbitrary non reflexive Banach spaces, see \cite{embry}.}. Consider thus the Borel operator 
\[
\mathcal{B}_2^\epsilon : L^2_{\beta,\Gamma_\beta B + \epsilon/2} \rightarrow \ell^2_{\beta,\Gamma_\beta B + \epsilon},
\]
as well as the inclusion 
\[
i_2^\epsilon : \ell^2_{\beta B} \rightarrow \ell^2_{\beta,\Gamma_\beta B+\epsilon}.
\]
From the inclusions \eqref{eq:embed_func} and \eqref{eq:embed_seq}, the operators $\mathcal{B}_2^\epsilon$ and $i_2^\epsilon$ are well-defined and bounded. Further, we have 
\[
\opRange i_\infty^\epsilon \subset \opRange i_2^\epsilon,\quad \opRange \mathcal{B}_2^\epsilon \subset \opRange \mathcal{B}_\infty^\epsilon,
\]
hence it is sufficient for us to show that
\[
\forall \epsilon > 0,\quad 
\opRange i_2^\epsilon \subset \opRange \mathcal{B}_2^\epsilon.
\]
We now fix $\epsilon > 0$ and aim  to show that 
\[
\opRange i_2^{2\epsilon} \subset \opRange \mathcal{B}_2^{2\epsilon},
\]
which is enough to conclude. For clarity we put $\mathcal{B} := \mathcal{B}_2^{2\epsilon}$ and $i := i_2^{2\epsilon}$. We identify $Y:= \ell^2_{\beta,\Gamma_\beta B + 2\epsilon}$ to its dual space, do the same for $\ell^2_{\beta B}$, and consider the true dual of $X:= L^2_{\beta,\Gamma_\beta B + \epsilon}$ for the duality arguments. This means that 
\[
\mathcal{B} : X \rightarrow Y,\quad i : \ell^2_{\beta B} \rightarrow Y,\quad \mathcal{B}^* : Y \rightarrow X',\quad i^* : Y \rightarrow \ell^2_{\beta B}.
\]
From Douglas's lemma \cite{douglas}, the range inclusion
\[
 \opRange i \subset \opRange \mathcal{B},   
\]
is equivalent to
\begin{equation}\label{eq:incl_adj}
\exists c > 0,\quad \forall a \in Y,\quad \| \mathcal{B}^* a \|_{X'} \geq c \| i^*a \|_{\ell^2_{\beta B}},    
\end{equation}
which is what we aim at showing until the end of the proof. From the density of the set of finitely supported sequences in $Y$, which we denote by $\mathbb{C}^{(\mathbb{N})}$, it is enough to show \eqref{eq:incl_adj} when $a \in \mathbb{C}^{(\mathbb{N})}$. For such $a$ consider the concentrated functional 
\[
f:= \sum_{n=0}^\infty a_n \delta^{(n)},\quad \langle \delta^{(n)} , \phi \rangle := \phi^{(n)}(0).
\]
Clearly, $f \in X'$, and we observe that for any $\phi \in X$
\begin{align*}
\langle f , \phi \rangle_{X',X} &= \sum_{n=0}^\infty a_n \phi^{(n)}(0) \\
&= \sum_{n=0}^\infty \frac{[(\Gamma_\beta B + 2\epsilon)^nn^{n\beta}]^2 
a_n \cdot (\mathcal{B}\phi)_n}{[(\Gamma_\beta B + 2\epsilon)^nn^{n\beta}]^2} \\
&= \langle \tilde{a} , \mathcal{B}\phi \rangle_{Y} \\
&= \langle \mathcal{B}^* \tilde{a} , \phi \rangle_{X',X},
\end{align*}
where 
\[
\tilde{a}_n:= [(\Gamma_\beta B + 2\epsilon)^nn^{n\beta}]^2   a_n.
\]
In particular, $f = \mathcal{B}^* \tilde{a}$. From 
\[
\forall b \in Y,\quad \forall n \in \mathbb{N},\quad (i^*b)_n = \left( \frac{B}{\Gamma_\beta B+2\epsilon} \right)^{2n} b_n,
\]
we see that \eqref{eq:incl_adj} is equivalent to 
\[
\sum_{n=0}^\infty  |B^n n^{n\beta}a_n|^2 \leq c \|f\|_{X'}^2,
\]
with a constant $c > 0$ not depending on $a$. We then invoke Lemma \ref{lem:mitjagin} for the concentrated functional $f$. To this end, observe that 
\[
C^\infty_{\beta,\Gamma_\beta B + \epsilon/2} \subset L^2_{\beta,\Gamma_\beta B+\epsilon},
\]
with continuous dense inclusion. Hence, by transposition, we obtain
\[
f \in X' \hookrightarrow (C^\infty_{\beta,\Gamma_\beta B + \epsilon/2})',
\]
so that Lemma \ref{lem:mitjagin} applies to $f$. Plain computations show that the function $\Phi$ constructed at Lemma \ref{lem:mitjagin} satifies
\[
\Phi(s) = \sum_{n=0}^\infty a_n  (is)^n,
\]
so that the estimate \eqref{eq:growth_coef} brings
\[
|a_n| \leq C \|f\|_{X'} \left( \frac{1+\delta}{\Gamma_\beta B+ \epsilon/2}\Gamma_\beta\right)^n n^{-n\beta},
\]
for a constant $C = C(\delta,\beta,B,\epsilon)$, where $\delta > 0$ is to be tuned later on. We deduce
\[
|B^n n^{n\beta}a_n| \leq C \|f\|_{X'} \left( \frac{1+\delta}{1+\frac{\epsilon}{2\Gamma_\beta B}} \right)^n,
\]
and observe that as $\beta,B,\epsilon$ are fixed, we can take $\delta > 0$ small enough so that the right-hand side is summable. 
\end{proof}
\begin{Remark}
Douglas's lemma yields a slightly better conclusion than the inclusion \eqref{eq:incl_adj}: there exists a bounded operator $\mathfrak{R} : \ell^2_{\beta B} \rightarrow L^2_{\beta,\Gamma_\beta B + \epsilon}$ such that $i = \mathcal{B} \mathfrak{R}$. In other words, the interpolation problem in the Gevrey class of order $s$ can be solved continuously and linearly. Compare with \cite[Proposition 21, \S 7, p. 119]{mitjagin_bases}.
\end{Remark}
To establish the optimality of the factor $\Gamma_\beta$ we collect some results on concentrated functionals. Our starting point is the following observation: let $f$ be a concentrated functional on $C^\infty_{\beta B}$ for some $B > 0$. It is tempting to write 
\begin{equation}\label{eq:repr_dirac}
    f = \sum_{n=0}^\infty f_n \delta^{(n)},
\end{equation}
where $(f_n)$ are the coefficients of the entire function $\Phi$ built in Lemma \ref{lem:mitjagin}. However, for a sequence $(a_n)$ of complex numbers and $\phi \in C^\infty_{\beta B}$, we are only able to make the series $\sum a_n \phi^{(n)}(0)$ convergent under the assumption
\[
\sum_{n=0}^\infty |a_n| B^n n^{n\beta} < \infty.
\] 
So from \eqref{eq:growth_coef}, we are only able to make $\sum f_n \delta^{(n)}$ converge in  $(C^\infty_{\beta A})'$, for $A < B/\Gamma_\beta$. This suggests that \eqref{eq:repr_dirac} is not the optimal representation of $f$, and also note that we have not proven yet that this is a valid representation. 
\begin{Proposition}\label{prop:concentrated_func} Fix $B > 0$ and $\beta > 1$. We introduce a function $h$ in  $C^\infty_{\gamma A}$ for some $1 < \gamma < \beta$ and $A > 0$, that is constant to 1 around 0 and flat at $\pm 1$. We consider the weight 
\[
\omega(\xi) := \exp\left( \frac{\beta}{e B^{1/\beta}} |\xi|^{1/\beta}\right),
\]
and a function $\Psi \in L^2(\mathbb{R}, \omega^{-2}(\xi) d\xi)$, that is moreover entire of order $\rho < 1$ on the complex plane. Then we have
\begin{enumerate}
\item The relation
\begin{equation}\label{eq:def_concentrated_fourier}
\langle g , \phi \rangle:= \int_\mathbb{R} \Psi(\xi) \mathcal{F}[h \phi](\xi)d\xi
\end{equation}
defines a concentrated functional $g$ on $L^2_{\beta B}$, which does not depend on $h$. Moreover, we have the estimate
\begin{equation}\label{eq:fourier_isom}
    c \| g \|_{(L^2_{\beta B})'} \leq \| \Psi \|_{L^2(\mathbb{R}, \omega^{-2}(\xi) d\xi)} \leq C \| g \|_{(L^2_{\beta B})'},
\end{equation}
where the constants $c,C$ do not depend on $\Psi$. 
\item Let $f$ be a concentrated functional on $L^2_{\beta B}$. The function $\Phi$ constructed in Lemma \ref{lem:mitjagin} lies in $L^2(\mathbb{R}, \omega^{-2}(\xi) d\xi)$. Moreover, the functional $g$ defined by \eqref{eq:def_concentrated_fourier} with $\Psi = \Phi$ equals $f$, and for every $A < B/\Gamma_\beta$ there holds 
\[
f = \sum_{n=0}^\infty f_n (-i)^n\delta^{(n)}\quad \mbox{on}\quad C^\infty_{\beta A}.
\]
\end{enumerate} 
\end{Proposition}
\begin{proof}
\underline{Step 1:} We show that \eqref{eq:def_concentrated_fourier} defines a functional on $L^2_{\beta B}$. For a function $\phi \in C^\infty_{\beta B}$, the function $\varphi:= h \phi$ lies again in $L^2_{\beta B}$ and satisfies 
\[
\| \varphi \|_{L^2_{\beta B}} \lesssim \| \phi \|_{L^2_{\beta B}},
\]
see \cite[Lemma 3.7]{reachable_martin}. Moreover, since $h$ is compactly supported, the function $\varphi$ can be extended to the whole real line, satisfying the same Gevrey estimate. More precisely, we have 
\[
\| \varphi \|_{L^2_{\beta B}} = \| \varphi \|_{\mathcal{G}_{\beta, \hat{B},0}},\quad \hat{B} := \frac{\beta}{e B^{1/\beta}}.
\]
We deduce that 
\[
\| \varphi \|_{\mathcal{G}_{\beta, \hat{B} ,0}} = \| \varphi \|_{L^2_{\beta B}} \lesssim \| \phi \|_{L^2_{\beta B}},
\]
and from Theorem \ref{theo:plancherel_gevrey} we find that 
\[
\|\mathcal{F} \varphi \|_{L^2(\mathbb{R}, \omega^{2}(\xi) d\xi)} = \|\mathcal{F} \varphi \|_{\hat{\mathcal{G}}_{\beta, \hat{B} ,0}} \lesssim \| \varphi \|_{\mathcal{G}_{\beta, \hat{B} ,0}} \lesssim \| \phi \|_{L^2_{\beta B}}.
\]
From the Cauchy-Schwarz inequality, we deduce that the integral \eqref{eq:def_concentrated_fourier} thus converges, that \eqref{eq:def_concentrated_fourier} defines a functional on $L^2_{\beta B}$, and the first estimate in \eqref{eq:fourier_isom} follows. \newline
\newline
\underline{Step 2:} From now on we assume that the order $\rho$ of $\Psi$ satisfies $\rho < 1/\gamma$ (it will be proved in Step $4$ that we have indeed $\rho\leqslant 1/\beta$). For the moment, we can assume that $\rho < 1/\gamma$ by choosing such $\gamma$ and fixing $h$ as in the statement of the Proposition. We show that 
\begin{equation}\label{eq:Diracs_g}
g = \sum_{n=0}^\infty a_n(-i)^n \delta^{(n)},\quad \mbox{on} \quad C^\infty_{\beta A}  
\end{equation}
holds for every $A < B/\Gamma_\beta$, where $(a_n)$ are the Taylor coefficients of $\Psi$. To this end, we fix $A < B/\Gamma_\beta$, and observe that both functionals involved in \eqref{eq:Diracs_g} are continuous over $C^\infty_{\beta A}$ where the set $C_\gamma^\infty$ is dense. Hence we only have to show the equality for test functions in $C_\gamma^\infty$. Let $\phi \in C^\infty_{\gamma D}$ for some $D > 0$ and compute 
\begin{align*}
\langle g , \phi \rangle &= \frac{1}{\sqrt{2 \pi}}\int_\mathbb{R} \Psi(\xi)\mathcal{F}[\phi h](\xi) d\xi \\
&= \frac{1}{\sqrt{2 \pi}}\int_\mathbb{R} \sum_{n=0}^\infty a_n \xi^n \mathcal{F}[\phi h](\xi) d\xi \\
&= \sum_{n=0}^\infty a_n \frac{1}{\sqrt{2 \pi}}\int_\mathbb{R}\xi^n \mathcal{F}[\phi h](\xi) d\xi \\
&= \sum_{n=0}^\infty a_n (-i)^n(\phi h)^{(n)}(0) \\
&= \sum_{n=0}^\infty a_n (-i)^n\phi^{(n)}(0),
\end{align*}
where at the third equality, the interversion of the series and the integral has to be justified. For $n \in \mathbb{N}$ we compute, with $\varphi:= h \phi$,
\begin{align*}
\int_\mathbb{R} |\xi^n \mathcal{F}\varphi(\xi)| d\xi &= \| \mathcal{F}\varphi^{(n)} \|_{L^1(\mathbb{R})} \\
&\lesssim \| \varphi^{(n)} \|_{H^1(\mathbb{R})} \\
&\lesssim \| \varphi^{(n)} \|_{C^1[-1,1]} \\
&\lesssim D^n n^{n\gamma} + D^{n+1} (n+1)^{(n+1)\gamma}.
\end{align*}
Because $\Psi$ has order $< 1/\gamma$, for any fixed $\lambda>\gamma$ we have
\[
\sum_{n=0}^\infty |a_n| n^{n \lambda} < \infty,
\]
see \cite[Theorem 2.2.2]{boas}. Hence 
\[
\sum_{n=0}^\infty |a_n| \int_\mathbb{R}|\xi^n \mathcal{F}[\phi h](\xi)| d\xi \lesssim \sum_{n=0}^\infty |a_n|\left[D^n n^{n\gamma} + D^{n+1} (n+1)^{(n+1)\gamma} \right] \lesssim \sum_{n=0}^\infty |a_n| n^{n \lambda} < \infty.
\]
This shows the equality \eqref{eq:Diracs_g}. \newline
\newline
\underline{Step 3:} Let us show that $g$ is a concentrated functional on $L^2_{\beta B}$. Take $\phi \in L^2_{\beta B}$ which vanishes identically on $[-\delta,\delta]$ for some $\delta > 0$. We introduce a mollifying kernel $(\varrho_\epsilon)_{\epsilon > 0}$ that is Gevrey of order $\gamma$ and put $\psi_\epsilon := (h\phi)*\varrho_\epsilon$. Plain computations show that $\phi_\epsilon \in C^\infty_{\gamma}$ together with $\psi_\epsilon \rightarrow h\phi$ in $L^2_{\beta B}$, so that from Theorem \ref{theo:plancherel_gevrey} we also have $\mathcal{F}\psi_\epsilon \rightarrow \mathcal{F}(h\phi)$ in $L^2(\mathbb{R}, \omega^{2}(\xi) d\xi)$. We observe that 
\[
\opsupp \phi_\epsilon \subset [-1,-\delta + \epsilon] \cup [\delta - \epsilon , 1],
\]
which is distant to $0$ for small enough $\epsilon$. For such $\epsilon$, a computation identical to that made at Step 2 brings 
\[
\frac{1}{\sqrt{2 \pi}}\int_\mathbb{R} \Psi(\xi)\mathcal{F}\psi_\epsilon(\xi) d\xi = \sum_{n=0}^\infty a_n (-i)^n\psi_\epsilon^{(n)}(0) = 0,
\]
and passing to the limit $\epsilon \rightarrow 0$ (which is valid since $\mathcal{F}\psi_\epsilon \rightarrow \mathcal{F}(h\phi)$ in $L^2(\mathbb{R}, \omega^{2}(\xi) d\xi)$ and $\Psi \in L^2(\mathbb{R}, \omega^{-2}(\xi) d\xi)$),
we find that 
\[
0 = \frac{1}{\sqrt{2 \pi}}\int_\mathbb{R} \Psi(\xi)\mathcal{F}[h\phi](\xi) d\xi = \langle g , \phi \rangle,
\]
which shows that $g$ is concentrated. 

In particular, owing to the inclusion $C^\infty_{\beta B} \subset L^2_{\beta B}$, the functional $g$ is a concentrated functional on $C^\infty_{\beta B}$ and we are allowed to consider the entire function $\Phi$ constructed in Lemma \ref{lem:mitjagin}. Let us then show that $\Phi = \Psi$. We observe that for fixed $s \in \mathbb{C}$, the function $t \mapsto e^{ist}$ is analytic on the real line, hence we may test \eqref{eq:Diracs_g} against it. This yields 
\[
\Phi(s) = \langle g(t) , e^{ist} \rangle = \sum_{n=0}^\infty a_n s^n = \Psi(s). 
\]
\underline{Step 4:} The function $\Psi = \Phi$ therefore satisfies the estimates of Lemma \ref{lem:mitjagin}, and in particular it has order $\rho \leq 1/\beta$. This justifies that we may fix $1 < \gamma < \beta$, and therefore $h$, independently of $\Psi$. This also shows that $g$ does not depend on $h$. \newline
\newline
\underline{Step 5:} Let us show the second estimate in \eqref{eq:fourier_isom}. We compute 
\[
\|g\|_{(L^2_{\beta B})'} = \sup_{\| \phi \|_{L^2_{\beta B}} = 1} | \langle g , \phi \rangle | = \sup_{\| \phi \|_{L^2_{\beta B}} = 1} \frac{1}{\sqrt{2 \pi}} \left| \int_\mathbb{R} \Psi(\xi) \mathcal{F}[h \phi](\xi)d\xi \right|,
\]
and acknowledge the following fact, which has already been used in this proof: if $\psi \in \mathcal{G}_{\beta, \frac{\beta}{e B^{1/\beta}},0}$, then its restriction to $[-1,1]$ lies in $L^2_{\beta B}$. To ease the notation we use the shorthand $\mathcal{G}$ for the previous space. Let us make the following computations that we will justify later on
\begin{align*}
    \sup_{\| \phi \|_{L^2_{\beta B}} = 1} \left| \int_\mathbb{R} \Psi(\xi) \mathcal{F}[h \phi](\xi)d\xi \right| &\geq  \sup_{\phi \in \mathcal{G} : \| \phi \|_{L^2_{\beta B}} = 1} \left| \int_\mathbb{R} \Psi(\xi) \mathcal{F}[h \phi](\xi)d\xi \right| \\
    &= \sup_{\phi \in \mathcal{G} : \| \phi \|_{L^2_{\beta B}} = 1} \left| \int_\mathbb{R} \Psi(\xi) \mathcal{F}\phi(\xi)d\xi \right| \\
    &= \sup_{\underset{\phi \not \equiv 0 ~on~ [-1,1]}{\phi \in \mathcal{G}}} \frac{1}{\|\phi\|_{L^2_{\beta B}}} \left| \int_\mathbb{R} \Psi(\xi) \mathcal{F}\phi(\xi)d\xi \right| \\
    &\geq \sup_{\underset{\phi \not \equiv 0 ~on~ [-1,1]}{\phi \in \mathcal{G}}} \frac{1}{\|\phi\|_\mathcal{G}} \left| \int_\mathbb{R} \Psi(\xi) \mathcal{F}\phi(\xi)d\xi \right| \\
    &\geq \sup_{\underset{\phi \not \equiv 0 ~on~ \mathbb{R}}{\phi \in \mathcal{G}}} \frac{1}{\|\phi\|_\mathcal{G}} \left| \int_\mathbb{R} \Psi(\xi) \mathcal{F}\phi(\xi)d\xi \right| \\
    &\gtrsim \| \Psi \|_{L^2(\mathbb{R}, \omega^{-2}(\xi) d\xi)}.
\end{align*}
The first equality is due to the fact that the functional 
\[
\tilde{g} : \phi \mapsto \int_\mathbb{R} \Psi(\xi) \mathcal{F}\phi(\xi)d\xi,
\]
is a concentrated functional on $\mathcal{G}$, which can be shown reasoning similarly as in Step 3. The second inequality is due to $\| \phi \|_{L^2_{\beta B}} \leq \| \phi \|_\mathcal{G}$. The third inequality is because $\tilde{g}$ is concentrated on $\mathcal{G}$. The last inequality is due to the Plancherel-Gevrey Theorem \ref{theo:plancherel_gevrey}, which asserts that the Fourier transform is a topological isomorphism $\mathcal{G} \rightarrow L^2(\mathbb{R}, \omega^{-2}(\xi) d\xi)$. This shows the second estimate in \eqref{eq:fourier_isom}. \newline
\newline
\underline{Step 6:} To end the proof of the Proposition, it remains to show its second item. For, let $f$ be a concentrated functional on $L^2_{\beta B}$, going back to the proof of Lemma \ref{lem:mitjagin} we see that $\Phi$ satisfies the following estimate on the real line 
\begin{equation}\label{eq:estimate_PHI_real}
    \forall \delta > 0,\quad \exists C_\delta > 0,\quad \forall \xi \in \mathbb{R},\quad |\Phi(\xi)| \leq C_\delta \exp\left( \frac{\beta}{e} \left( \frac{1+\delta}{B}|\xi|\right)^{1/\beta}\right).
\end{equation}
Thus, $\Phi$ is nearly in the weighted Lebesgue space $L^2(\mathbb{R}, \omega^{-2}(\xi) d\xi)$. In fact, for fixed $0 < \tilde{B} < B$, consider the weight
\[
\tilde{\omega}(\xi) := \exp\left( \frac{\beta}{e \tilde{B}^{1/\beta}} |\xi|^{1/\beta}\right).
\]
From \eqref{eq:estimate_PHI_real} we deduce that $\Phi \in L^2(\mathbb{R}, \tilde{\omega}^{-2}(\xi) d\xi)$, hence $\tilde{g}$ defined by the relation \eqref{eq:def_concentrated_fourier} with $\Psi = \Phi$ defines a concentrated functional on $L^2_{\beta \tilde{B}}$. Clearly, the functional $\tilde{g}$ does not depend on $\tilde{B}$, hence we denote it by $g$. For $\phi \in C^\infty_{\beta \tilde{B}}$ we compute
\begin{align*}
\langle g , \phi \rangle &= \frac{1}{\sqrt{2\pi}} \int_\mathbb{R} \Phi(\xi) \mathcal{F}[\phi h](\xi) d\xi \\
&= \frac{1}{\sqrt{2\pi}}\int_\mathbb{R} \langle f(t) , e^{i\xi t} \rangle \mathcal{F}[h\phi](\xi)d\xi \\
&= \left\langle f(t),\frac{1}{\sqrt{2\pi}}\int_\mathbb{R} e^{i\xi t} \mathcal{F}[h\phi](\xi)d\xi \right\rangle \\
&= \langle f , \phi h \rangle = \langle f , \phi \rangle.
\end{align*}
This shows that $g = f$ on $C^\infty_{\beta \tilde{B}}$. We deduce that $g$ is a concentrated functional on $L^2_{\beta B}$, so that $\Phi$ lies indeed in $L^2(\mathbb{R}, \omega^{-2}(\xi) d\xi)$. The representation of $f$ as a series of Diracs follows from Step 2, which ends the proof. 
\end{proof}
We are now in position to show the optimality of the factor $\Gamma_\beta$.
\begin{Theorem}
For every $\epsilon,B > 0$ and $\beta > 1$, there exists a sequence $(a_n) \in K(n^{n\beta} B^n)$ which can be interpolated by no $\phi \in C^\infty_{\beta,\Gamma_\beta B - \epsilon}$. 
\end{Theorem}
\begin{proof}
Assume by contradiction that this is not true, then fix $\epsilon,B > 0$ and $\beta > 1$ so that 
\[
\opRange \mathcal{B}_\infty^\epsilon \supset \opRange i_\infty^\epsilon,
\]
where 
\[
\mathcal{B}_\infty^\epsilon: C^\infty_{\beta,\Gamma_\beta B-\epsilon} \rightarrow K(n^{n\beta} (\Gamma_\beta B -\epsilon)^n),\quad i_\infty^\epsilon: K(n^{n\beta} B^n) \rightarrow K(n^{n\beta} (\Gamma_\beta B -\epsilon)^n).
\]
Similarly as in the proof of Theorem \ref{theo:mitjagin1}, we deduce that
\begin{equation}\label{eq:range_incl_cette_fois_clabonne}
    \opRange \mathcal{B} \supset \opRange i,
\end{equation}
where 
\[
\mathcal{B}: L^2_{\beta,\Gamma_\beta B - \epsilon/2} \rightarrow \ell^2_{\beta,\Gamma_\beta B},\quad i: \ell^2_{\beta,B} \rightarrow \ell^2_{\beta,\Gamma_\beta B}.
\]
Denote $X = L^2_{\beta,\Gamma_\beta B - \epsilon/2}$, from \eqref{eq:range_incl_cette_fois_clabonne} and Douglas's Lemma we have 
\begin{equation}\label{eq:douglas_quantitative}
    \exists c > 0,\quad \forall a \in \ell^2_{\beta,\Gamma_\beta B},\quad \|\mathcal{B}^*a\|_{X'} \geq c \|i^*a\|_{\ell^2_{\beta B}}.
\end{equation}
Fix $f$ a concentrated functional on $L^2_{\beta B}$ and consider the weight 
\[
\omega(\xi) := \exp\left( \frac{\beta}{e \tilde{B}^{1/\beta}} |\xi|^{1/\beta}\right), \quad \tilde{B} := \Gamma_\beta B - \epsilon/2.
\]
From Proposition \ref{prop:concentrated_func}, the entire function $\Phi$ constructed in Lemma \ref{lem:mitjagin} belongs to $L^2(\mathbb{R},\omega^{-2}(\xi)d\xi)$, and it has order $\leq 1/\beta$ on the complex plane. From \cite[Theorem 3.1]{Borichev}, there exists a sequence $(P_j)$ of polynomials such that $P_j \rightarrow \Phi$ for the $L^2(\mathbb{R},\omega^{-2}(\xi)d\xi)$-norm. For every $j$, we put 
\[
P_j(s) =: \sum_{n=0}^\infty f_n^j s^n,\quad \hat{P}_j := \sum_{n=0}^\infty f_n^j (-i)^n \delta^{(n)},\quad a_n^j := n^{2n\beta}(\Gamma_\beta B)^{2n}f_n^j(-i)^n,
\]
so that $\hat{P}_j = \mathcal{B}^*a^j$. We apply \eqref{eq:douglas_quantitative} to the sequence $a^j$, which is finitely supported, to obtain 
\begin{equation}\label{eq:estimate_j}
    \sup_{n \in \mathbb{N}} n^{n\beta}B^n|f_n^j| \leq \sqrt{\sum_{n=0}^\infty |n^{n\beta}B^nf_n^j|^2} = \| i^* a^j \|_{\ell^2_{\beta B}} \leq \frac{1}{c} \| \mathcal{B}^*a^j \|_{X'} \lesssim \| P_j \|_{L^2(\mathbb{R},\omega^{-2}(\xi)d\xi)}.
\end{equation}
Because $(P_j)$ is a convergent sequence of $L^2(\mathbb{R},\omega^{-2}(\xi)d\xi)$, right-hand side above is bounded with respect to $j$. Because $P_j$ converges to $\Phi$ for the $L^2(\mathbb{R},\omega^{-2}(\xi)d\xi)$-norm we deduce that $\hat{P}_j$ converges to $f$ for the $X'$-norm, and in particular we have 
\[
\forall n \in \mathbb{N},\quad f_n^j \xrightarrow[j \rightarrow \infty]{} f_n.
\]
Passing to the limit in \eqref{eq:estimate_j} we deduce that $(f_n)$ satisfies 
\begin{equation}\label{eq:coef_estimate}
    \exists c > 0,\quad \forall n \in \mathbb{N},\quad |f_n| \leq c B^nn^{-n\beta},
\end{equation}
an improvement over \eqref{eq:growth_coef} from Lemma \ref{lem:mitjagin}.

Now consider an arbitrary entire function 
\[
\Psi(z) = \sum_{n=0}^\infty a_n z^n,
\]
of order $<1$ on $\mathbb{C}$, satisfying 
\begin{equation}\label{eq:growth_Phi_real1}
\exists C> 0,\quad \forall \xi \in \mathbb{R},\quad |\Psi(\xi)| \leq C \exp\left( \frac{\beta}{e A^{1/\beta}}  |\xi|^{1/\beta} \right),
\end{equation}
for some $A > \Gamma_\beta B - \epsilon/2$. From such $\Psi$ we construct a concentrated functional $g$ on $L^2_{\beta A}$ as in Proposition \ref{prop:concentrated_func}. From the previously established fact, the coefficients $(a_n)$ of $\Psi$ satisfy the estimate \eqref{eq:coef_estimate}. In particular, we have 
\[
\exists C > 0,\quad \forall z \in \mathbb{C},\quad |\Psi(z)| \leq C \exp \left( \frac{\beta}{e B^{1/\beta}} |z|^{1/\beta} \right).
\]
Put now $A := \Gamma_\beta B - \epsilon/4 > \Gamma_\beta B - \epsilon/2$, so that we have established the following fact: any function $\Psi$ that is entire and of order $<1$ on $\mathbb{C}$, that has order $\leq 1/\beta$ and type no greater than
\[
\sigma:= \frac{\beta}{e(\Gamma_\beta B - \epsilon/4)^{1/\beta}},
\]
on the real line, is in fact of order $\leq 1/\beta$ and type no greater than
\[
\sigma':= \frac{\beta}{e B^{1/\beta}},
\]
on the whole complex plane. Observe that 
\[
\lambda:= \frac{\sigma'}{\sigma} = \left( \Gamma_\beta - \frac{\epsilon}{4 B} \right)^{1/\beta} < \Gamma_\beta^{1/\beta}.
\]
By a scaling argument we find that: any function $\Psi$ that is entire and of order $<1$ on $\mathbb{C}$, and that has order $\leq 1/\beta$ and type $\leq K$ on the real line, is of order $\leq 1/\beta$ and type $\leq \lambda K$ on $\mathbb{C}$. To arrive to a contradiction we consider the Mittag-Leffler function 
\[
E_\beta(z) = \sum_{k=0}^\infty \frac{z^k}{\Gamma(\beta k + 1)},
\]
which is entire on $\mathbb{C}$. From Stirling's formula it is clear that $E_\beta$ has order $1/\beta$ and type 1 on $\mathbb{C}$ (see \cite[\S 2.2]{boas}). Moreover, the function $E_\beta$ has the asymptotics
\[
|E_\beta(iy)| \sim \frac{1}{\beta} \exp\left( \cos\left( \frac{\pi}{2 \beta} \right) |y|^{1/\beta} \right),\quad \mathbb{R} \ni y \rightarrow \pm \infty,
\]
see \cite[Proposition 3.6]{mittag_leffler}, hence it has type $\cos( \pi / 2\beta)$ on the imaginary axis. Taking $\Psi(s) = E_\beta(is)$ and invoking the previously deduced fact we obtain: $E_\beta$ has order $1/\beta$ and type $\lambda \cos(\pi / 2\beta)$ on $\mathbb{C}$. Because 
\[
\lambda \cos\left( \frac{\pi}{2 \beta} \right) < \Gamma_\beta^\beta \cos\left( \frac{\pi}{2 \beta} \right) = 1,
\]
we arrive to a contradiction, as $E_\beta$ has type $1$ on $\mathbb{C}$. 
\end{proof}

\section{Proof of the Lemma on the discrete version of the Laplace method}\label{sec:appB}
This section is devoted to the proof of Lemma \ref{lem:laplace_discrete}. The proof involves several computations so we divide it into several steps to ease the reading. 

\underline{Step 0:} There is no loss of generality in assuming that $u(x_0) = 0$, and thus we aim  to show that 
\[
\sum_{k=0}^n e^{-nu(k/n)} \sim \sqrt{\frac{2 \pi n}{u''(x_0)}},\quad n \rightarrow \infty. 
\]

\underline{Step 1:} For every $a > 0$ and $b \in \mathbb{R}$, we claim that 
\begin{equation}\label{flp1}
\sum_{k=-\infty}^{+\infty} \exp\left[ -n \frac{a}{2}\left( \frac{k}{n} - b\right)^2\right] = \sqrt{\frac{2n\pi}{a}} \left[ 1 + O \left( e^{- \frac{2n \pi^2}{a}}\right)\right],
\end{equation}
where the $O$ term is uniform with respect to the parameters $n \geq 1$, $b \in \mathbb{R}$, and $a > \underline{a}$ where $\underline{a} > 0$ is any given positive number. To show this, we follow the proof of \cite[Lemma 2]{laplace_discrete}. We write 
\[
\sum_{k=-\infty}^{+\infty} \exp\left[ -n \frac{a}{2}\left( \frac{k}{n} - b\right)^2\right] = e^{-n \frac{ab^2}{2}} \sum_{k=-\infty}^{+\infty} e^{-\frac{a}{2n}k^2}e^{abk} = e^{-n \frac{ab^2}{2}} \Theta \left( \frac{ab}{2i\pi} , - \frac{a}{2in\pi}\right),
\]
where $\Theta$ is Jacobi's theta function, defined by
\[
\Theta(z,\tau) = \sum_{k=-\infty}^{+\infty} e^{i\pi k^2\tau} e^{2i\pi k z}, \quad z \in \mathbb{C},\quad \opIm \tau > 0.
\]
The functional equation for $\Theta$ gives 
\[
\Theta\left( z , -\frac{1}{\tau}\right) = \sqrt{\frac{\tau}{i}} e^{i\pi \tau z^2} \Theta(z\tau,\tau),
\]
see \cite[Theorem 7.1]{mumford}. Applying this formula with $\tau = 2in\pi /a$ and $z = ab/2i\pi$ yields 
\[
    e^{-n \frac{ab^2}{2}} \Theta\left( \frac{ab}{2i\pi} , - \frac{a}{2in\pi}\right) = e^{-n \frac{ab^2}{2}} \Theta \left( z , - \frac{1}{\tau} \right) = \sqrt{\frac{2n\pi}{a}} \sum_{k=-\infty}^{+\infty} e^{-\frac{2 \pi^2n}{a}k^2} e^{2i\pi kbn}.
\]
In the series of the right-hand side, the term $k = 0$ is dominant:
\[
\left| \sum_{k \neq 0} e^{-\frac{2 \pi^2n}{a}k^2} e^{2i\pi kbn} \right| \leq 2 \sum_{k=1}^\infty e^{-\frac{2 \pi^2n}{a}k^2} \leq C e^{- \frac{2 \pi^2n}{a}},
\]
for some constant $C >0$ independent of $n \geq 1$ and $a > \underline{a}$. This shows the claim.

\underline{Step 2:} We fix $\delta > 0$, $n \geq 1$ and we put 
\[
m(\delta) = \inf_{|x-x_0| \geq \delta} u(x),
\]
which is positive ($u$ has the unique global minimum 0). We deduce that 
\begin{equation}\label{flp4}
\sum_{\underset{\left| \frac{k}{n}-x_0\right| \geq  \delta}{k=0}}^n e^{-nu(k/n)}= O\left( n e^{-nm(\delta)}\right).
\end{equation}
We then deduce, thanks to \eqref{flp4}, that
\begin{equation}\label{flp3}
\sum_{k=0}^n e^{-nu(k/n)} = \sum_{\underset{\left| \frac{k}{n}-x_0\right| < \delta}{k=0}}^n e^{-nu(k/n)} + \sum_{\underset{\left| \frac{k}{n}-x_0\right| \geq  \delta}{k=0}}^n e^{-nu(k/n)} =  \sum_{\underset{\left| \frac{k}{n}-x_0\right| < \delta}{k=0}}^n e^{-nu(k/n)} + O\left( n e^{-nm(\delta)}\right),
\end{equation}
and estimate the sum appearing on the right-hand side. We estimate it from above, and we will estimate it from below in the next step. We put 
\[
\ell(\delta) = \inf_{|x-x_0| < \delta} u''(x),
\]
which is well defined and positive for small enough $\delta$, what will be assumed henceforth, and which converges to $u''(x_0)$ as $\delta \rightarrow 0^+$. This allows
\[
|x-x_0| < \delta \Longrightarrow u(x) \geq \frac{\ell(\delta)}{2}(x-x_0)^2,
\]
hence, by \eqref{flp1}, 
\begin{equation}\label{flp5}
\sum_{\underset{\left| \frac{k}{n}-x_0\right| < \delta}{k=0}}^n e^{-nu(k/n)} \leq \sum_{\underset{\left| \frac{k}{n}-x_0\right| < \delta}{k=0}}^n e^{-n\frac{\ell(\delta)}{2}\left( \frac{k}{n} - x_0\right)^2} \leq \sum_{k=-\infty}^{+\infty} e^{-n\frac{\ell(\delta)}{2}\left( \frac{k}{n} - x_0\right)^2} = \sqrt{\frac{2n\pi}{\ell(\delta)}} \left[ 1 + O \left( e^{- \frac{2n \pi^2}{\ell(\delta)}}\right)\right],
\end{equation}
where the $O$ term is uniform in $n \geq 1$ and $\delta > 0$ small enough. Thus, for $\delta > 0$ fixed and small enough, we have 
\[
\limsup_{n \rightarrow \infty} \sqrt{\frac{u''(x_0)}{2n\pi}}\sum_{k=0}^n e^{-nu(k/n)} \leq \sqrt{\frac{u''(x_0)}{\ell(\delta)}}.
\]
Letting $\delta \rightarrow 0^+$ in the right-hand side, we find that
\begin{equation}\label{flp2}
\limsup_{n \rightarrow \infty} \sqrt{\frac{u''(x_0)}{2n\pi}}\sum_{k=0}^n e^{-nu(k/n)} \leq 1.
\end{equation}

\underline{Step 3:} The estimate from below is slightly more tedious. We put, for $\delta > 0$ small enough so that $[x_0-\delta,x_0+\delta] \subset (0,1)$,
\[
L(\delta) = \sup_{|x-x_0| < \delta} u''(x) > 0,
\]
so that by \eqref{flp1},
\begin{align*}
    \sum_{\underset{\left| \frac{k}{n}-x_0\right| < \delta}{k=0}}^n e^{-nu(k/n)} & \geq \sum_{\underset{\left| \frac{k}{n}-x_0\right| < \delta}{k=0}}^n e^{-n\frac{L(\delta)}{2}\left( \frac{k}{n} - x_0\right)^2} = \left\lbrace \sum_{k=-\infty}^{+\infty} - \sum_{k=n+1}^\infty - \sum_{k=-\infty}^{-1} - \sum_{\underset{\left| \frac{k}{n}-x_0\right| \geq \delta}{k=0}}^n \right\rbrace e^{-n\frac{L(\delta)}{2}\left( \frac{k}{n} - x_0\right)^2} \\
    &= \sqrt{\frac{2n\pi}{L(\delta)}} \left[ 1 + O \left( e^{- \frac{2n \pi^2}{L(\delta)}}\right)\right] - \left\lbrace  \sum_{k=n+1}^\infty +\sum_{k=-\infty}^{-1} \right\rbrace e^{-n\frac{L(\delta)}{2}\left( \frac{k}{n} - x_0\right)^2} \\
    &\quad  + O\left(ne^{-n \frac{L(\delta)}{2}\delta^2 }\right). \numberthis{\label{eq:estimate_below}}
\end{align*}
We are thus left to estimate the remaining two series in the right-hand side of \eqref{eq:estimate_below}. Until the end of the proof, $\kappa$ is any positive number which depends on $\delta$ and $x_0$, changing line to line, small enough so that the written asymptotics are valid. To bound above the first series we use a comparison between series and integrals:
\[
    \sum_{k=n+1}^\infty e^{-n\frac{L(\delta)}{2}\left( \frac{k}{n} - x_0\right)^2} \leq \int_n^\infty e^{-\kappa n\left( \frac{x}{n} - x_0\right)^2} dx = \sqrt{\frac{n}{\kappa}} \int_{(1-x_0) \sqrt{n\kappa}}^\infty e^{-x^2}dx = O( e^{-\kappa n}).
\]
For the other series we compute 
\[
\sum_{k=-\infty}^{-1} e^{-n\frac{L(\delta)}{2}\left( \frac{k}{n} - x_0\right)^2} \leq \sum_{k=1}^\infty \exp\left[ - \kappa \left( \frac{k^2}{n} + 2kx_0 + x_0^2n\right)\right] \leq e^{-\kappa n} \sum_{k=1}^\infty e^{-\kappa k} = O(e^{-\kappa n}).
\]
Coming back to \eqref{eq:estimate_below} we find 
\[
\sum_{\underset{\left| \frac{k}{n}-x_0\right| < \delta}{k=0}}^n e^{-nu(k/n)} \geq \sqrt{\frac{2n\pi}{L(\delta)}} \left[ 1 + O( e^{- \kappa n})\right] + O(e^{-\kappa n}) \geq \sqrt{\frac{2n\pi}{L(\delta)}}+ O(e^{-\kappa n}),
\]
and we conclude, similarly as in the previous step, that 
$$
\liminf_{n \rightarrow \infty} \sqrt{\frac{u''(x_0)}{2n\pi}}\sum_{k=0}^n e^{-nu(k/n)} \geq 1,
$$
which, together with \eqref{flp2}, concludes the proof of the Lemma.
\printbibliography[heading = bibintoc]
\end{document}